\documentclass[review]{elsarticle}
 \usepackage{lineno,hyperref}
\modulolinenumbers[5]

\journal{Journal of Computational and Applied Mathematics }
\usepackage{array,multirow,makecell}
\usepackage{url}

\usepackage{doi}
\setcellgapes{1pt}
\makegapedcells
\newcolumntype{R}[1]{>{\raggedleft\arraybackslash }b{#1}}
\newcolumntype{L}[1]{>{\raggedright\arraybackslash }b{#1}}
\newcolumntype{C}[1]{>{\centering\arraybackslash }b{#1}}
\usepackage{tabularx}
\usepackage {float}
\usepackage{xcolor}
\usepackage[utf8]{inputenc}
\usepackage{graphicx}
\usepackage[francais,english]{babel}
\usepackage[T1]{fontenc}
\usepackage{charter}
\usepackage{setspace}
\usepackage{amsfonts}
\usepackage{amsmath}
\usepackage{amssymb}
\usepackage{amsthm}
\usepackage{fancybox}
\usepackage{afterpage}
\usepackage{geometry}
\usepackage{multicol}
\usepackage[inline,shortlabels]{enumitem}
\hypersetup{colorlinks,linkcolor={red},citecolor={blue},urlcolor={red}}  
\usepackage[tikz]{bclogo}
\providecommand{\U}[1]{\protect\rule{.1in}{.1in}}
\graphicspath{{./graphics/}}
\providecommand{\U}[1]{\protect\rule{.1in}{.1in}}
\providecommand{\U}[1]{\protect\rule{.1in}{.1in}}
\theoremstyle{plain}
\usepackage{tikz} 
\newtheorem{Theorem}{Theorem} [section]

\newtheorem{Corollary}{Corollary}[section]
\newtheorem{Proposition}{Proposition}[section]
\newtheorem{Lemma}{Lemma}[section]
\newtheorem{Remark}{Remark}[section]
\newtheorem{Example}{Example}[section]
\newtheorem{Proof}{Proof}[section]

\newcommand{\R}{\mathbb{R}}

\newcommand{\V}{\ \ \ \ \ \ }

\everymath{\displaystyle}

\numberwithin{equation}{section}
\begin{document}

	\afterpage{\newgeometry{top=2cm, bottom=2.5cm, outer=5.5cm, inner=2cm, heightrounded, marginparwidth=3.5cm, marginparsep=0.4cm, includeheadfoot}}

	\newgeometry{left=1.5cm, right=1.5cm, top=1.9cm,bottom=1.2cm}
	 
	\thispagestyle{empty}

\begin{frontmatter}

\title{A new Relaxation Method for Optimal Control of Semilinear Elliptic Variational Inequalities Obstacle Problems}


\author[mymainaddress,mysecondaryaddress]{El Hassene Osmani\corref{mycorrespondingauthor}}
\cortext[mycorrespondingauthor]{Corresponding author}
\ead{el-hassene.osmani@insa-rennes.fr}

\author[mysecondaryaddress]{ Mounir Haddou}
\ead{mounir.haddou@insa-rennes.fr}

\author[mymainaddress]{Naceurdine Bensalem}
\ead{naceurdine.bensalem@univ-setif.dz}

\address[mymainaddress]{University Ferhat Abbas of Setif 1,
Faculty of Sciences, Laboratory of Fundamental and Numerical Mathematics, Setif 19000,  Algeria}
\address[mysecondaryaddress]{ University of  Rennes, INSA Rennes, CNRS, IRMAR - UMR 6625, F-35000 Rennes, France}

\begin{abstract}
In this paper, we investigate optimal control problems governed by  semilinear elliptic variational inequalities involving constraints on the state, and more precisely the obstacle problem. Since we adopt a numerical point of view, we first relax the feasible domain of the problem, then  using both mathematical programming methods and penalization methods we get optimality conditions with smooth Lagrange multipliers. Some numerical experiments using IPOPT algorithm are presented to verify the efficiency of our approach.
\end{abstract}


\end{frontmatter}

		\section{INTRODUCTION }

	In this paper, we investigate optimal control problems where the state is described by semilinear variational inequalities. These problems involve state constraints as well. We use the method of  \cite{Bergounioux-1}  to obtain a generalization of the results of the quoted paper to the semilinear case. It is known that Lagrange multipliers may not exist for such problems  \cite{Bergounioux-Mignot}. Nevertheless, providing qualifications conditions, one can exhibit multipliers for relaxed problems. These multipliers usually allow to get optimality conditions of Karush-Kuhn-Tucker type. Our purpose is to get optimality conditions that are useful from a numerical point of view. Indeed, we have to ensure the existence of Lagrange multipliers to prove the convergence of lagrangian methods and justify their use. These kind of problems have been extensively studied by many authors, see for instance  \cite{Barbu,Friedman,Mignot-Puel} .   
	
		\vspace{0.4cm}
	
The variational inequality will be interpreted as a state equation, introducing another
control function as in  \cite{Bergounioux-1} . Then, the optimal control problem may be considered as a "standard" control problem governed by a semilinear partial differential equation, involving pure and mixed control-state constraints which are not necessarily convex. In order to derive some optimality conditions, we have to "relax"   the domain; so we do not solve the original problem but this point of view will be justied and commented. Then, the use of Mathematical Programming in Banach spaces methods  \cite{Troltzsch,Zowe}   and penalization techniques provides first-order necessary optimality conditions.
 	\vspace{0.4cm}
 
The first part of this paper is devoted to the presentation of the problem: we recall some classical results on variational inequalities there. In sections $ 3 $ we give approximation formulations of the original problem. In section  $ 4 $ we briefly present some Mathematical Programming results in Banach spaces. Next, we use a penalization technique and apply the tools of the previous section to the penalized problem. We obtain penalized optimality conditions, and assuming some qualification conditions we may pass to the limit to get optimality conditions for the original problem. In the last section, we present some numerical results  and propose  conclusion.

	\vspace{0.4cm}
	\section{PROBLEM SETTING }

Let  $\Omega$ be an open, bounded subset of $\R^{n}$  with a smooth boundray   $  \partial\Omega $.  We shall denote $\mid\mid.\mid\mid_{V}$, the norm in Banach space $ V,$ and more precisely $\mid\mid.\mid\mid $ the $L^{2}(\Omega)$-norm. In the same way, $~  \langle .,.\rangle $ denotes the duality product between   $ H^{-1}( \Omega) $ and $  H_{0}^{1}(\Omega)$, we will denote similarly the $L^{2}(\Omega)$-scalar product when there is no ambiguity. Let us set

	\begin{equation} \label{eq1.1}	
K=\lbrace   y ~  \vert ~ y\in H_{0}^{1}(\Omega), ~   y\geq \psi  ~ a.e.~ \text{in}  ~ \Omega \rbrace, 
\end{equation}
where $ \psi $ is a $ H^{2}(\Omega)\cap H_{0}^{1}(\Omega) $ function.
\vspace{0.3cm}

In the squel $ g $ is a non decreasing, $ C^{1} $ real-valued function such that $ g^{'} $ is bounded, locally Lipschitz continuous and $ f  $ belongs to $ L^{2}(\Omega) $. Moreover, $ U_{ad} $ is a non empty, closed and convex subset of $ L^{2}(\Omega) $. \\
   For each $ v $ in $ U_{ad} $ we consider the following variational inequality problem : find $ y\in K$ such that 

\begin{equation} \label{eq1.2}
  a(y,z)+G(y)-G(z)\geq  \langle\ v+f,y-z \rangle  \quad   \forall z\in K, 
 \end{equation}
	where $ G $	is a primitive function of $ g $, and $ a $  is a bilinear form defined on $ H_{0}^{1}(\Omega)\times H_{0}^{1}(\Omega) $ by 
	
	\begin{equation} \label{eq1.3}
	 a(y,z)=\sum_{i,j=1}^{n}\int_{\Omega}a_{ij}\frac{\partial y}{\partial x_{i}}\frac{\partial z}{\partial x_{j}}~dx+\sum_{i=1}^{n}\int_{\Omega}b_{i}\frac{\partial y}{\partial x_{i}}z ~dx + \int_{\Omega}cyz ~dx,
	 \end{equation}
		
where $ a_{ij}, b_{i},c $ belong to $ L^{\infty}(\Omega)$. Moreover, we assume that  $a_{ij}$	 belongs to $ \mathcal{C}^{0,1}(\bar{\Omega}) $ (the space of Lipschitz continuous functions in $ \Omega $) and that $ c $ is nonnegative. The bilinear form $ a(., .) $ is continuous on \\ $ H_{0}^{1}(\Omega)\cap H_{0}^{1}(\Omega):$ 
\begin{equation} \label{eq1.4}
 \exists M >0, ~\forall (y,z)\in H_{0}^{1}(\Omega)\cap H_{0}^{1}(\Omega) ~~~~~ a(y,z)\leq M\mid\mid y\mid\mid_{ H_{0}^{1}(\Omega)} ~\mid\mid z \mid\mid_{H_{0}^{1}(\Omega)}  
 \end{equation}
and is coercive : 
\begin{equation} \label{eq1.5}
\exists\delta > 0, ~\forall  y\in H_{0}^{1}(\Omega),~~~~~~ a(y,y)\geq \delta \mid\mid y \mid\mid_{H_{0}^{1}(\Omega)}^{2}. 
  \end{equation}
We set $ A $ the elliptic differential operator from $  H_{0}^{1}(\Omega)$ to $ H^{-1}(\Omega) $  defined by 
\\
\begin{center}
 $  \forall(z,v)\in H_{0}^{1}(\Omega) \times H_{0}^{1}(\Omega) ~~~~ \langle\ Ay,z \rangle=a(y,z).$ 
\end{center}

For any $ v\in $  $L^{2}(\Omega) $,  problem  \eqref{eq1.2} has a unique solution $ y=y\lbrack v \rbrack \in H_{0}^{1}(\Omega) $, since the coercivity of the problem in $ y $, and $ v $. As the obstacle function belongs to $  H^{2}(\Omega)$ we have an additional regularity result : $ y\in  H^{2}(\Omega)\times H_{0}^{1}(\Omega)$ (see ~\cite{Barbu-1, Bergounioux}). Moreover \eqref{eq1.2} is equivalent to (see  \cite{Mignot-Puel}) 
\begin{equation} \label{eq1.6}
  Ay+g(y)=f+v+\xi,~ y\geq\psi, ~\xi\geq 0,~\langle \xi,y-\psi \rangle=0,  
\end{equation}
	where $ "\xi\geq 0" $ stands for $ "\xi(x)\geq 0 $ almost everywhere on $~ \Omega $". The above equation can be viewent as the  optimality system for problem \eqref{eq1.2} : $ \xi $ is the multiplier associated to the contraint $ y\geq \psi $. 
	It is a priori an element of $ H^{-1}(\Omega) $ but the regularity result for $ y $ shows that $ \xi\in L^{2}(\Omega) $, so that $  \langle \xi,y-\psi \rangle_{ H^{-1}(\Omega)\times L^{2}(\Omega) } =\langle\ \xi,y-\psi\rangle $.   \\
	\begin{Remark} 
	Applying the simple transformation $ y^{*}=y-\psi $, we may assume that $\psi=0  $.
	Of course functions $ g $ and $ f $ are modified as well, but this shift preserves their generic proprerties ( local lipschitz-continuity, monotonicity).
	\end{Remark}
	In the sequel $ g $ is non decreasing, $ \mathcal{C}^{1} $ real-valued function such that 
	\begin{equation} \label{eq1.6}
	  \exists \gamma \in \R,~\exists \beta \geq 0 ~~ \text{such that} ~~\forall y \in \R ~~  \vert g(y)\vert \leq \gamma +\beta \vert y \vert .
	\end{equation}
	
	We denote similarly the real valued function $ g $ and the Nemitsky operator such that $g(y)(x)=g(y(x))  $ for every $ x\in \Omega. $
	\newpage
	
	 Therefore we keep the same notations. Now, let us consider the optimal control problem defined as follows :
	\begin{center}
	$ \text{min} \left\{J(y,v)\stackrel{def}{=}\frac{1}{2}\int_{\Omega}(y-z_{d})^{2}dx+\frac{\nu}{2}\int_{\Omega}(v-v_{d})^{2}dx  ~\vert ~y=y\lbrack v \rbrack,~ v \in U_{ad},~ y\in K  \right\}$,

	\end{center}
	where $z_{d}, v_{d}\in L^{2}(\Omega)  $ and $ \nu >0 $ are given quantities. \\
	This problem is equivalent to the problem governed by a state equation (instead of inequality) with mixed state and control constraints :
	\begin{center}
	   $\text{min} \left\{J(y,v)=\frac{1}{2}\int_{\Omega}(y-z_{d})^{2}dx+\frac{\nu}{2}\int_{\Omega}(v-v_{d})^{2}dx\right\},  $ \hspace{10mm} $(\mathcal{P})  $ \end{center}
	   \vspace{0.2cm}
	   \begin{equation} \label{eq1.7}
	 Ay+g(y)=f+v+\xi  ~~  \text{in}~ \Omega, ~~  y=0 \quad \text{on} \quad  \partial\Omega ,   
	 \end{equation}
	\begin{equation} \label{eq1.8}
	  (y,v,\xi)\in \mathcal{D}, 
	 \end{equation}

	where 
		\begin{equation} \label{eq1.9}
 \mathcal{D} =  \{ (y,v,\xi)\in H_{0}^{1}(\Omega)\times L^{2}(\Omega)\times L^{2}(\Omega)~\vert~ v \in U_{ad},~ y \geq 0,~ \xi \geq 0,~ \langle\ y,\xi\rangle =0 \}.
	 \end{equation}

	We assume that the feasible set $ \tilde{\mathcal{D}}=\{(y,v,\xi)\in \mathcal{D} ~\vert ~$ relation \eqref{eq1.7} is satisfied\} is non empty. We know, then that problem $ (\mathcal{P}) $ has at least an optimal solution (not necessarily unique) that we shall denote $ (\bar{y},\bar{v},\bar{\xi}), $ since the coercivity of the problem in $ y $, and $ v $ see for instance \cite{Mignot-Puel} .   \\
	Similar problems  have been studied also in  \cite{Bergounioux-Tiba} but in the convex context ($ \mathcal{D} $ is convex). Here, the main difficulty comes from the fact that the feasible domain $ \mathcal{D} $ is non-convex and has an empty relative interior because of the bilinear constraint $ "\langle\ y,\xi \rangle =0". $ \\
	So, we cannot use generic convex analysis methods that have been used for instance in  \cite{Bergounioux-Tiba}. To derive optimality conditions in this case, we are going to use methods adapted to quite general mathematical programming. Unfortunately, the domain $ \mathcal{D} $ (i.e. the constraints set ) does not satisfy the usual (quite weak) assumption of mathematical programming theory. This comes essentially from the fact that $ L^{\infty}$-interior of $ \mathcal{D} $ is empty. \\
	So we  cannot ensure the existence of Lagrange multipliers. This problem does not satisfy classical constraint qualifications  (in the usual KKT sense). One can find several counter-examples in finite and infinite dimension in  \cite{Bergounioux-Mignot} . \\ 
	\section{A RELAXED PROBLEM}
	In order to "relax" the complementarity constraint "$ \langle\ y,\xi\rangle =0 $" we introduce a family of $ \mathcal{C}^{1} $ functions $ \theta_{\alpha}:\R^{+} \to [0, 1[,$ ($  \alpha >0$) with the following properties (see  \cite{HADDOU} for more precision on these smoothing functions):
	\begin{enumerate}[label=(\roman*)]
           \item $ \forall \alpha >0,$  $ \theta_{\alpha}$ is nondecreasing, concave and  $ \theta_{\alpha}(1)<1,$ 
          \item $ \forall \alpha >0 ~~~~ ~~ \theta_{\alpha}(0)=0, $
           \item $  \forall x >0 \quad  \lim\limits_{ \alpha \rightarrow 0} \theta_{\alpha}(x) = 1~~~~ ~~$ and $~~ \lim\limits_{ \alpha \rightarrow 0} \theta_{\alpha}^{'}(0)>0. $
           
   \end{enumerate}

\begin{Example}
	The function below satisfy assumption $ (i-iii) $ ( see  \cite{HADDOU}):
	\begin{center}
	$  \theta_{\alpha}^{1}(x)=\frac{x}{x+\alpha},$\\
	\vspace{0.2cm}
	$  \theta_{\alpha}^{W}(x)= 1-e^{-\frac{x}{\alpha}}$,
	
	\vspace{0.2cm}
	$  \theta_{\alpha}^{\text{log}}(x)= \frac{\text{log}(1+x)}{\text{log}(1+x+\alpha)}$.
	
	\end{center}
	
\end{Example}
	Functions $ \theta_{\alpha} $ are built to approximate the complementarity constraint in the following sense :
	\begin{center}
	$ \forall (x, y) \in \R \times \R \quad xy=0 \tilde{\Longleftrightarrow} \theta_{\alpha}(x)+\theta_{\alpha}(y) \leq 1 $ for $ \alpha $ small enough.
	\end{center}
	\newpage
	More precisely, we have the following proposition.
	
	\begin{Proposition}
	Let $ (y, v, \xi) \in \mathcal{D}$ and $ \theta_{\alpha}^{1} $ satisfying  $ (i-iv) $. Then 
	\begin{center}
	$ \langle\ y, \xi \rangle=0  \Longrightarrow \theta_{\alpha}^{1}(y)+ \theta_{\alpha}^{1}(x) \leq 1  ~~~~~~~~ \text{a.e. in}~ \Omega. $ 
	\end{center}

	\end{Proposition}
	
	The proof of the  proposition it is based on the followings lemmas :
	
	\begin{Lemma}
	For any  $ \varepsilon  > 0 $, and $ x,y\geq 0 $, there exists $  \alpha_{0} >0$ such that 
	
	\begin{center}
	$\forall \alpha \leq \alpha_{0}, ~~~(\text{min}(x,y)=0)\Longrightarrow (\theta_{\alpha}(x)+\theta_{\alpha}(y)\leq 1)\Longrightarrow (min(x,y)\leq \varepsilon) $.
	
	\end{center}
	
	\end{Lemma}
	
	\begin{Proof}
The first property is obvious since $ \theta_{\alpha}(0)=0 $ and $ \theta_{\alpha} \leq 1. $ \\
Using assumtion $(iii)$ for $ x=\varepsilon, $ we have \\
\begin{center}
$  \forall r>0, \quad   \exists \alpha_{0}>0 ~\vert ~ \forall \alpha \leq \alpha_{0} \quad 1-\theta_{\alpha}(\varepsilon) < r,$
\end{center}
	
	so that, if we suppose that $ \text{min}(x,y)>\varepsilon, $ assumption $(i)$ gives 
	\begin{center}
	$  \forall r>0,\quad \theta_{\alpha}(x)+\theta_{\alpha}(y)> 2\theta_{\alpha}(\varepsilon)>2(1-r). $
	\end{center}
	Then if we choose $ r <\frac{1}{2}, $ we obtain that $ \theta_{\alpha}(x)+\theta_{\alpha}(y)>1. $
	\end{Proof}
	  \begin{Lemma}
	we have 
	\begin{center}
	 \begin{enumerate}[(1)]
	\item $ \forall x \geq 0,~ \forall y \geq 0   ~~ $   $ \theta_{\alpha}^{1}(x)+ \theta_{\alpha}^{1}(y)\leq 1 \Longleftrightarrow  x.y\leq \alpha^{2},~~  \text{and}   $ 
	\item   $ \forall x \geq 0,~ \forall y \geq 0 ~~ $ $ x.y=0 \Longrightarrow  \theta_{\alpha}^{\geq 1}(x)+ \theta_{\alpha}^{\geq 1}(y)\leq 1 \Longrightarrow x.y\leq \alpha^{2},  $
	\end{enumerate}	
	\end{center}
	where $ \theta_{\alpha}^{\geq 1} $ verifying $(i-iv) $ and  $ \theta_{\alpha}^{\geq 1}  \geq \theta_{\alpha}^{ 1}$.
	\end{Lemma}
	\begin{Proof}
	$(1)$ We have 
	\begin{center}
	$ \theta_{\alpha}^{1}(x)+ \theta_{\alpha}^{1}(y) = \frac{2xy+\alpha x+\alpha y}{xy+\alpha x+\alpha y+\alpha^{2}},$
	\end{center}
	so that 
\begin{equation}\label{eq1.10}
 \begin{split}	 
\theta_{\alpha}^{1}(x)+ \theta_{\alpha}^{1}(y)\leq 1  & \Longleftrightarrow 2xy+\alpha x+\alpha y \leq xy+\alpha x+\alpha y+\alpha^{2} \\
&	 \Longleftrightarrow  x.y \leq \alpha^{2}.
\end{split}	 
\end{equation}
	The first part of $ (2) $ follows obviously form Lemma 3.1 and the second one is a direct consequence of $ (1) $ since 
	\begin{center}
	$ \theta_{\alpha}^{\geq 1}(x)+ \theta_{\alpha}^{\geq 1}(y)\leq 1 \Longrightarrow \theta_{\alpha}^{1}(x)+ \theta_{\alpha}^{1}(y)\leq 1.  $
	\end{center}
	\end{Proof}

	More precisely, we consider the domain $ \mathcal{D}_{\alpha} $ instead of $ \mathcal{D} $, with $ \alpha >0 $   (using the function $\theta_{\alpha}^{1})   $  we obtain :  
	\begin{equation} \label{eq1.11}
	 \mathcal{D}_{\alpha}=  \left\{ (y,v,\xi)\in H_{0}^{1}(\Omega)\times L^{2}(\Omega)\times L^{2}(\Omega)~\vert ~ v\in U_{ad},~ y\geq 0,~ \xi \geq 0, ~\frac{y}{y+\alpha}+\frac{\xi}{\xi+\alpha}\leq 1,~ a.e. ~in ~\Omega \right\}.
	\end{equation}
We may justify and motivate this points of view numerically, since it is usually not possible to ensure $  "\langle\ y,\xi \rangle =0"$  during a computation but rather $ "\frac{y}{y+\alpha}+\frac{\xi}{\xi+\alpha}\leq 1" $ where $ \alpha $ is a prescribed tolerance : it may be chosen small as wanted, but strictly positive. \\
So the problem turns to be qualified if the bilinear constraint $  "\langle\ y,\xi \rangle =0"$ is relaxed to $ "\frac{y}{y+\alpha}+\frac{\xi}{\xi+\alpha}\leq 1" $ a.e. in $ \Omega $.

	In the sequel, we consider an optimal control problem $ (\mathcal{P}^{\alpha}) $ where the feasible domain is $ \mathcal{D}_{\alpha} $ instead of $ \mathcal{D}  $.
	\newpage
	 Moreover, we must add a bound constraint on the control $ \xi $ to be able to ensure the existence of a solution of  this relaxed problem. More precisely we consider : 
	
\begin{center}
	$$
	(\mathcal{P}^{\alpha}) ~~~  \left\{
	\begin{array}{llllll} 
	\text{min}~ J(y,v) ~~~~~~~~~~~~~~~~~~~~~~~~~~~~~~~~~~~~~~~~~~~~  \\
	Ay+g(y)= f+v+\xi ~ \text{in} ~ \Omega, ~y \in H_{0}^{1}(\Omega),\\
	(y,v,\xi)\in \mathcal{D}_{\alpha, R} 
	\end{array}
	\right.
	$$
\end{center}

		where $ R >0 $ may be very large and 
		\begin{center}
		$  \mathcal{D}_{\alpha, R}=\{(y,v,\xi)\in \mathcal{D}_{\alpha}~\vert ~  \mid\mid \xi\mid\mid_{L^{2}(\Omega)} ~ \leq R \} $.  
		\end{center}

	From now on, we omit the index $ R $ since this constant is definitely fixed, such that 
	\begin{equation} \label{eq1.12}
	  R  \geq   ~\mid\mid \bar{\xi}\mid\mid_{L^{2}(\Omega)},  
	\end{equation}
	(we recall that $ (\bar{y},\bar{v},\bar{\xi}) $ is a solution of $( \mathcal{P}) $).\\
	We will denote $\mathcal{D}_{\alpha}:=\mathcal{D}_{\alpha, R}$, and $ V_{ad}=\{\xi\in L^{2}(\Omega) ~\vert ~ \xi\geq 0, \mid\mid \xi\mid\mid_{L^{2}(\Omega)}\leq R \} $.$ V_{ad}$ is obviously a closed, convex subset of $ L^{2}(\Omega) $ \\
	As  $ (\bar{y},\bar{v},\bar{\xi}) \in  \mathcal{D}  $, we see (with \eqref{eq1.12})   that $ \mathcal{D}_{\alpha} $ is non empty for any $\alpha >0$.
	\subsection{Existence Result}
	In order to prove an existence result for $(\mathcal{P}^{\alpha})$, we state first a basic but essential lemma.  
	
	\begin{Lemma}
	Assume that $ (y_{n},v_{n}) $ is a bounded sequence in $H_{0}^{1}(\Omega) \times L^{2}(\Omega)$ such that 
	$\xi_{n} \mathrel{\mathop:}=Ay_{n}+g(y_{n})-f-v_{n} $ is bounded in $ L^{2}(\Omega) $. Then, one may extract subsequences  (still denoted similarly) such that

	\begin{itemize}
  \item $ v_{n}$ converges weakly  to some   $ \tilde{v} $ in $ L^{2}(\Omega) $,
  \item $ y_{n}$  converges strongly to some   $ \tilde{y} $ in $ H_{0}^{1}(\Omega) $,
    \item $ g(y_{n})$   converges strongly to  $ g(\tilde{y}) $ in $ L^{2}(\Omega) $,
      \item $ Ay_{n}+g(y_{n})-f-v_{n}$  converges weakly to   $ A\tilde{y}+g(\tilde{y})-f-\tilde{v} $ in $ L^{2}(\Omega). $
\end{itemize}
\end{Lemma}

 \begin{Proof}

 Let $ (y_{n},v_{n}) $ be a bounded sequence in $H_{0}^{1}(\Omega) \times L^{2}(\Omega)$; therefore $ (y_{n},v_{n}) $ weakly converges to some $ (\tilde{y},\tilde{v}) $ in $H_{0}^{1}(\Omega) \times L^{2}(\Omega)$ (up to a subsequence). Similarly, $\xi_{n}$ weakly converges to some  $ \tilde{\xi} $ in $  L_{2}(\Omega)$. Thanks to  \cite{Krasnosel-Rutickii}  (Theorem 17.5, p174), assumption \eqref{eq1.6} yields that 
\begin{center}
$ (y_{n})_{n\geq 0} $ bounded in $ L^{2}(\Omega) \Longrightarrow (g(y_{n}))_{n\geq 0} $ bounded in $ L^{2}(\Omega) $.
\end{center}
	
	As, $ y_{n} $ weakly converges to $ \tilde{y} $ in  $ H_{0}^{1}(\Omega) $, it strongly converges in $ L^{2}(\Omega) $ a.e in $ \Omega $. As $ g $ is continuous, $ g(y_{n}) $ converges a.e. in $\Omega  $  as well (up to subsequences).
	We conclude then (Lebesgue theorem), that $ g(y_{n}) $ strongly converges to $ g(\tilde{y})$ in   $  L^{2}(\Omega) $. \\
	
	Moreover when  $Ay_{n}=-g(y_{n})+f+v_{n}+\xi_{n}  $ is bounded in $L^{2}(\Omega)$ it will and  converge weakly to some  $ \tilde{z} $ in $L^{2}(\Omega)$. As $ y_{n} $ weakly converges to  $ \tilde{y} $ in $ H_{0}^{1}(\Omega) $, then $ Ay_{n} $ converges to $ A\tilde{y} $ in $ H^{-1}(\Omega) $, so $ \tilde{z} =A\tilde{y}$ and $ Ay_{n}$ weakly converges to $A\tilde{y}$ in $ L^{2}(\Omega) $ as well. 
Therefore  $ Ay_{n}$ strongly converges to $A\tilde{y}$ in $ H^{-1}(\Omega) $. \\
	Finally we get the weak convergence of $ Ay_{n}+g(y_{n})-f-v_{n}$ to $ A\tilde{y}+g(\tilde{y})-f-\tilde{v} $ in $L^{2}(\Omega)$ and the strong convergence of $ y_{n} $ to $\tilde{y }$ in $ H_{0}^{1}(\Omega) $.    
	  \end{Proof}

	  So that, we can consider that problem $( \mathcal{P}^{\alpha} )$ is a  "good" approximation of the original problem  $( \mathcal{P} )$ in tn the foloowing sense : 
	   
	  \begin{Theorem}
      
      For any  $\alpha > 0 $, $( \mathcal{P}^{\alpha} ) $ has at least one optimal solution (denoted $ (y_{\alpha},v_{\alpha},\xi_{\alpha} )$). Moreover. when  $\alpha $ goes to $  0$, $ y_{\alpha} $ strongly converges to $ \tilde{y} $  in $ H_{0}^{1}(\Omega)~$ (up to a subsequence), $ v_{\alpha}$ strongly converges to $ \tilde{v} $  in $  L^{2}(\Omega)~$(up to a subsequence),  $ \xi_{\alpha}~ $  weakly converges to $ \tilde{\xi} $ in $  L^{2}(\Omega) $ (up to a subsequence),  where $(\tilde{y},\tilde{v},\tilde{\xi})$ is a solution of $(\mathcal{P})$.

  \end{Theorem}

	 \begin{Proof}  
	    Let  $(y_{n},v_{n},\xi_{n}) $ be a minimizing sequence such that  $ J(y_{n},v_{n}) $ converges to $d^{\alpha}$ = inf$( \mathcal{P}^{\alpha} ) $.
	    As $ J(y_{n},v_{n}) $ is bounded, there exists a constant $ C $ such that we have :
	    \begin{center}
	     $ \forall n \quad  \mid\mid v_{n} \mid\mid_{L^{2}(\Omega)}\leq C $.
	    \end{center}
	    So, we may extract a subsequence (denoted similarly) such that $ v_{n} $ converges to $ v_{\alpha} $ weakly in $L^{2}(\Omega)$ and strongly in $H^{-1}(\Omega) $. As $ U_{ad} $ is a closed convex set, it is weakly closed and $v_{\alpha}  $  $\in U_{ad} $.\\
	    	\vspace{0.4cm}
	    On the other hand, we have $Ay_{n}+g(y_{n})-f-v_{n}=\xi_{n}  $ and. So 
	    \begin{center}
	    $ \langle\ Ay_{n},y_{n} \rangle+\langle\ g(y_{n}),y_{n} \rangle =\langle\ f+v_{n},y_{n} \rangle +\langle\ y_{n},\xi_{n} \rangle. $
	    \end{center}
	   In view of Lemma 3.2, we have :
	   \begin{center}
	    $ \forall y_{n} \geq 0, ~ \forall \xi_{n} \geq 0~~~~~~~ $ $  \frac{y_{n}}{y_{n}+\alpha}+\frac{\xi_{n}}{\xi_{n}+\alpha}\leq 1 \Longleftrightarrow  y_{n}\xi_{n}\leq \alpha^{2},$ 
	   \end{center}
	    the integral by the two ways, gives 
	    \begin{center}
	    $  \frac{y_{n}}{y_{n}+\alpha}+\frac{\xi_{n}}{\xi_{n}+\alpha}\leq 1 \Longleftrightarrow  y_{n}\xi_{n}\leq \alpha^{2} \Longrightarrow  \langle\ y_{n},\xi_{n} \rangle \leq \alpha^{2}Area(\Omega). $ 
	   \end{center}	    
	    So 
	    \begin{center}
	    $ \langle\ Ay_{n},y_{n} \rangle+\langle\ g(y_{n}),y_{n} \rangle =\langle\ f+v_{n},y_{n} \rangle +\langle\ y_{n},\xi_{n} \rangle \leq \langle\ f+v_{n},y_{n} \rangle + \alpha^{2}Area(\Omega) $.
	    \end{center}
	    The monotonicity of $ g $ gives 
	    \begin{center}
	   $ \langle\ Ay_{n},y_{n} \rangle\leq \langle\ Ay_{n},y_{n} \rangle +\langle\ g(y_{n})-g(0),y_{n} \rangle\leq \langle\ f+v_{n}-g(0),y_{n} \rangle + \alpha^{2}Area(\Omega) $.
	    \end{center}
	    Using the coercivity of $ A $, we obtain 
	    \begin{center}
	  $   \delta \mid\mid y_{n}\mid\mid_{H_{0}^{1}(\Omega)}^{2}\leq  \mid\mid f+v_{n}-g(0)\mid\mid_{H^{-1}(\Omega)}^{2}\mid\mid y_{n}\mid\mid_{H_{0}^{1}(\Omega)}+\alpha^{2}Area(\Omega)\leq C\mid\mid y_{n}\mid\mid_{H_{0}^{1}(\Omega)}+\alpha^{2}Area(\Omega).$
	    \end{center}
	  This yields that $ y_{n} $ is bounded in $ H_{0}^{1}(\Omega)$, since $ \Omega  $ is bounded, so $ y_{n} $ converges to $ y_{\alpha} $ weakly in $ H_{0}^{1}(\Omega)$ and strongly in $L^{2}(\Omega)$. Moreover as $ y_{n}  \in K $, and $ K $ is a closed convex set, $ K $ is weakly closed and $ y_{\alpha}\in K.$
	We have assumed that $ V_{ad} $ is $ L^{2}(\Omega)\text{-bounded}$. So, we can apply Lemma 3.3, and obtain that $\xi_{n}$ weakly converges to  $\xi_{\alpha}=Ay_{\alpha}+g(y_{\alpha})-f-v_{\alpha} \in V_{ad} $ in $L^{2}(\Omega)$. \\
	\begin{Remark} 
	$~~ \xi_{n}=Ay_{n}+g(y_{n})-f-v_{n},$ weakly converges to $ \xi_{\alpha}=Ay_{\alpha}+g(y_{\alpha})-f-v_{\alpha} $ in $ H^{-1}(\Omega) $, Unfortunately the weak convergence of $ \xi_{n}$ to $ \xi_{\alpha} $ in $ H^{-1}(\Omega) $ is not sufficient to conclude. We need this sequence to converge weakly  in $ L^{2}(\Omega)$. That is the reason why we have bounded $ \xi_{n}$ in $ L^{2 }(\Omega)$.
	      \end{Remark}
	         \vspace{0.1cm}      
	 At last, $  \frac{y_{n}}{y_{n}+\alpha}+\frac{\xi_{n}}{\xi_{n}+\alpha}$ converges to  $ \frac{y_{\alpha}}{y_{\alpha}+\alpha}+\frac{\xi_{\alpha}}{\xi_{\alpha}+\alpha} $  because of the strong convergence of $ y_{n} $ in $  L^{2}(\Omega)$  and the weak convergence of $ \xi_{n} $ in $  L^{2}(\Omega)$ and we obtain $ \frac{y_{\alpha}}{y_{\alpha}+\alpha}+\frac{\xi_{\alpha}}{\xi_{\alpha}+\alpha}\leq 1 $ : we just proved that $(y_{\alpha},v_{\alpha},\xi_{\alpha})\in \mathcal{D}_{\alpha}.$ 
	 The weak convergence and the lower semi-continuity of $ J $ give :
	 \begin{center}
	 $ d^{\alpha}=\displaystyle{\lim_{n \to \infty}}\text{inf} ~   J(y_{n},v_{n})\geq J(y_{\alpha},v_{\alpha})\geq d^{\alpha}.$ 
	 \end{center}
	    So $J(y_{\alpha},v_{\alpha})=d^{\alpha} $ and $(y_{\alpha},v_{\alpha},\xi_{\alpha})  $ is a solution of $ (\mathcal{P}^{\alpha}) $.
	\begin{itemize}
  \item Now, let us prove the second part of the theorem. First we note that  $ (\bar{y},\bar{v},\bar{\xi}) $ belongs to $ \mathcal{D}^{\alpha}  $ for any $ \alpha >0 $. So : 
     \begin{equation} \label{eq1.13}
	 \forall \alpha >0  ~~ ~~ J(y_{\alpha},v_{\alpha})\leq J(\bar{y},\bar{v}) < +\infty . 
	\end{equation}
\end{itemize}    
 and $ v_{\alpha} $ and $ y_{\alpha} $ are bounded respectively in  $ L^{2}(\Omega) $ and $ H^{1}_{0}(\Omega)$. Indeed, we use the previous arguments since $ v_{\alpha}$ is bounded in  $ L^{2}(\Omega) $ and  	 
 \begin{center}
	  $   \delta \mid\mid y_{\alpha}\mid\mid_{H_{0}^{1}(\Omega)}^{2} ~\leq ~ \mid\mid f+v_{\alpha}-g(0)\mid\mid_{H^{-1}(\Omega)}^{2}\mid\mid y_{\alpha}\mid\mid_{H_{0}^{1}(\Omega)}+\alpha^{2}Area(\Omega)\leq C\mid\mid y_{\alpha}\mid\mid_{H_{0}^{1}(\Omega)}+\alpha^{2}Area(\Omega).$
	    \end{center}  
	So (extracting a subsequence) $ v_{\alpha}$ weakly converges to some $ \tilde{v} $ in $ L^{2}(\Omega) $ and $ y_{\alpha}$ converges to some  $ \tilde{y} $ weakly in   $ H_{0}^{1}(\Omega) $ and strongly in  $ L^{2}(\Omega) $. As above, it is easy to see that $ \xi_{\alpha} $ weakly converges to $ \tilde{\xi}=A\tilde{y}+g(\tilde{y})-f-\tilde{v}  $ in $ L^{2}(\Omega) $ (Thanks Lemma 3.3), and that $ \tilde{y} \in K, ~\tilde{v} \in U_{ad}, ~ \tilde{\xi} \in V_{ad}$. 
	In the same way $ \frac{y_{\alpha}}{y_{\alpha}+\alpha}+\frac{\xi_{\alpha}}{\xi_{\alpha}+\alpha} $ converges to $ \frac{\tilde{y} }{\tilde{y} +\alpha}+\frac{\tilde{\xi} }{\tilde{\xi} +\alpha} $. As $ 0 \leq \frac{y_{\alpha}}{y_{\alpha}+\alpha}+\frac{\xi_{\alpha}}{\xi_{\alpha}+\alpha} \leq 1 $, from Lemma 3.2 we get : 
	\begin{center}
	$ 0 \leq ~~\frac{y_{\alpha}}{y_{\alpha}+\alpha}+\frac{\xi_{\alpha}}{\xi_{\alpha}+\alpha} ~~\leq 1 \Longleftrightarrow 0\leq  y_{\alpha}\xi_{\alpha} \leq \alpha^{2},$
	\end{center}
at the limit as $ \alpha \searrow $ 0 this implies that $ \tilde{y}\tilde{\xi} =0 \Longleftrightarrow \langle\ \tilde{y},\tilde{\xi}   \rangle =0.$ So ($\tilde{y},~\tilde{v},~\tilde{\xi})  \in \mathcal{D}$. This yields that      
	     \begin{equation} \label{eq1.14}
	    J(\bar{y},\bar{v}) \leq J(\tilde{y},\tilde{v}) .
	    \end{equation}
	    Once again, we may pass to the inf-limite in \eqref{eq1.13} to obtain :
	    \begin{center}
	 $ J(\tilde{y},\tilde{v}) $ $ \leq $ $ \displaystyle{\lim_{\alpha \to 0}} \text{inf} ~ J(y_{\alpha},v_{\alpha}) $ $ \leq J(\bar{y},\bar{v}). $
	    \end{center}
	This implies that 
	\begin{center}
	 $  J(\tilde{y},\tilde{v})= J(\bar{y},\bar{v}),  $
	    \end{center} 
	    therefore  $ (\tilde{y},\tilde{v},\tilde{\xi}) $ is a solution of ($ \mathcal{P}$). Moreover, as $ \displaystyle{\lim_{\alpha \to 0}}J(y_{\alpha},v_{\alpha}) = J(\tilde{y},\tilde{v}) $ and $ y_{\alpha} $ strongly converges to $ \tilde{y} $ in $ L^{2}(\Omega) $, we get $ \displaystyle{\lim_{\alpha \to 0}}\mid\mid v_{\alpha}\mid\mid_{L^{2}(\Omega)}= \mid\mid\tilde{v}\mid\mid_{L^{2}(\Omega)} $, so that $ v_{\alpha} $ strongly  converges to $ \tilde{v} $ in $ L^{2}(\Omega) $.\\
	    We  already know that $ \xi_{\alpha} $ weakly converges to $ \tilde{\xi} $ in $ L^{2}(\Omega). $ So $ \xi_{\alpha}+v_{\alpha}-g(y_{\alpha})+f=Ay_{\alpha} $ converges to $ \tilde{\xi}+\tilde{v}-g(\tilde{y})+f=A\tilde{y} $ weakly in $ L^{2}(\Omega) $ and strongly in $ H^{-1}(\Omega) $. As $ A $ is an isomorphism from $ H^{1}_{0}(\Omega) $ to $ H^{-1}(\Omega) $ this yields that $ y_{\alpha} $ strongly converges to $ \tilde{y} $ in $ H^{1}_{0}(\Omega).$   
	  \end{Proof}

	  We see then, that solutions of problem $( \mathcal{P}^{\alpha} )$ are "good" approximations of the desired solution of problem $ (\mathcal{P}) $. \\
	  Now, we would like to derive optimality conditions for the problem $ (\mathcal{P}^{\alpha})$, for $ \alpha > 0.$   
	
	In the squel, we study the unconstrained control case: $ U_{ad}=L^{2}(\Omega) $. We first present some Mathematical Programming tools that allow to prove the existence of Lagrange multipliers.
	

	\section{THE MATHEMATICAL PROGRAMMING POINT OF VIEW}

	  \vspace{0.3cm}
	   The non convexity of the feasible domain, does not allow to use convex analysis to get the existence of Lagrange multipliers. So we are going to use quite general mathematical programming methods in Banach spaces and adapt them to our framework. \\ The following results are mainly due to Zowe and Kurcyusz  \cite{Zowe}  and Troltzsch \cite{Troltzsch}  and we briefly present them in the following.\\
	   \vspace{0.2cm}
	  Let us consider real Banach spaces $  \mathcal{X}$, $  \mathcal{U}$, $  \mathcal{Z}_{1}$, $ \mathcal{Z}_{2} $ and a convex closed "admissible" set $\mathcal{U}_{ad} \subseteq \mathcal{U} $. In $ \mathcal{Z}_{2} $ a convex closed cone $P$ is given so that  $ \mathcal{Z}_{2} $ is partially ordered by $ x \leq y  \Leftrightarrow  x-y \in  P$. We deal also with :\\
	  \vspace{0.2cm}
	  
	  \begin{center}
	  $ f :\mathcal{X} \times \mathcal{U} \to \R $, Fréchet-differentiable functional, \\
	   $ T :\mathcal{X} \times \mathcal{U} \to \mathcal{Z}_{1} $ and $ G :\mathcal{X} \times \mathcal{U} \to \mathcal{Z}_{2} $ continuously Fréchet-differentiable operators. \\
	  \end{center}
	 
	  Now, consider the mathematical programming problem defined by :
	  \begin{equation} \label{eq1.15}
	     \text {min}~ \{f(x,u)~ \vert ~ T(x, u)=0,~ G(x, u) \leq 0, ~u \in \mathcal{U}_{ad}\}.  
	  \end{equation}
	  \newpage
	We denote the partial Fréchet-derivative of $ f, T, $ and $ G $ with respect to $ x $  and $ u  $ by a corresponding index $  x$ or $ u $. We suppose that the problem  \eqref{eq1.15} has an optimal solution that we call $ (x_{0}, u_{0}), $ and we  introduce the sets :
	\begin{center}
	\vspace{0.2cm}
$ \mathcal{U}_{ad}(u_{0})=\{u\in \mathcal{U}~ \vert~ \exists \lambda \geq 0,~ \exists u^{*}\in \mathcal{U}_{ad},~ u=\lambda(u^{*}-u_{0})\} $, \\
		\vspace{0.2cm}
$ P(G(x_{0},u_{0}))=\{z \in \mathcal{Z}_{2}~\vert ~\exists \lambda\geq 0,~ \exists p\in -P,~ z=p-\lambda G(x_{0},u_{0})\} $, \\
		\vspace{0.2cm}
$P^{+}=\{y\in \mathcal{Z}_{2}^{*}~\vert ~ \langle\ y,p \rangle \geq 0,~ \forall p\in P \}   $. \\
	\end{center}
		
One may now announce the main result about the existence of optimality conditions. \\

\begin{Theorem}

Let $ u_{0} $ be an optimal control with corresonding optimal state $ x_{0} $ and suppose that the following regularity condition is fulfilled : \\

 \begin{equation}\label{eq1.16}
 \forall(z_{1},z_{2})\in \mathcal{Z}_{1}\times \mathcal{Z}_{2}  ~~~ \text{the system}    ~~~~~~~~~~     \left\{
		\begin{aligned}
		T^{'}(x_{0},u_{0})(x,u)& &=~&z_{1} \\
		G^{'}(x_{0},u_{0})(x,u)&-p&=~&z_{2}\\
		\end{aligned}
		 \right.
\end{equation}

\begin{center}
is solvable with $ (x,u,p)\in \mathcal{X}\times \mathcal{U}_{ad}(u_{0})\times P(G(x_{0},u_{0})).   $ 
\end{center}

Then a Lagrange multiplier $ (y_{1},y_{2})  \in \mathcal{Z}_{1}^{*}\times \mathcal{Z}_{2}^{*}$ exists such that \\
\vspace{0.3cm}
\begin{equation}\label{eq1.17}
f_{x}^{'}(x_{0},u_{0})+T_{x}^{'}(x_{0},u_{0})*y_{1}+G_{x}^{'}(x_{0},u_{0})*y_{2}=0,  
\end{equation}

\begin{equation}\label{eq1.18}
 \langle\ f_{x}^{'}(x_{0},u_{0})+T_{x}^{'}(x_{0},u_{0})*y_{1}+G_{x}^{'}(x_{0},u_{0})*y_{2}, u-u_{0} \rangle \geq 0, ~~\forall u\in \mathcal{U}_{ad}, 
\end{equation}

\begin{equation}\label{eq1.19}
 y_{2}\in P^{+}, \quad \langle\ y_{2},G(x_{0},u_{0}) \rangle =0.
\end{equation}
 \end{Theorem}

	  Mathematical programming theory in Banach spaces allows to study problems where the feasible domain is not convex: this precisely our case (and we cannot use the classical convex theory and the Gâteaux differentiability to derive some optimality conditions). The Zowe and Kurcyusz condition  \cite{Zowe}  is a very weak condition to ensure the existence of Lagrange multipliers. It is natural to try to see if this condition is satisfied for the original problem $ (\mathcal{P}) $ : unfortunately, it is impossible (see  \cite{Bergounioux-1995}) and this is another justification (from a theoretical point of view) of the fact that we have to take $ \mathcal{D}_{\alpha} $ instead of $ \mathcal{D}$. \\
	  On the other hand, if we apply the previous general result "directly" to $ (\mathcal{P}^{\alpha})$ we obtain a complicated qualification condition \eqref{eq1.16} which seems difficult to ensure. So we would rather mix these "mathematical-programming methods" with a penalization method in order to "relax" the state-equation as well and make the qualification condition weaker and simpler.\\
	  
	  \section{PENALIZATIN APPROACH}
	 
	   \subsection{The penalized problem}
	 One of the difficulties comes from the fact that we have a coupled system. It would be easier if we had only one condition. In order to split the different constraints and make them "independent", we penalize the state equation to obtain an optimization problem with non convex constraints. Then we apply previous method to get  optimality conditions for the penalized problem. Of course, we may decide to penalize the bilinear constraint instead of the state equation : this leads to the same results.\\
	 Moreover we focus on the solution $ (y_{\alpha},v_{\alpha},\xi_{\alpha})$, so, following Barbu   \cite{Barbu}, we add some adapted penalization terms to the objective functional $ J $. \\
	 \newpage
	 From nowon,  $\alpha>0$ is fixed , so we omit the index $ \alpha  $ when no confusion is possible. For any $ \varepsilon >0$ we define a penalized functional $J_{\varepsilon}^{\alpha} $ on $ (H^{2}(\Omega)\cap H_{0}^{1}(\Omega)) \times L^{2}(\Omega)\times L^{2}(\Omega))  $ as following :\\

	  \begin{equation} \label{eq1.20}
		J_{\varepsilon}^{\alpha}(y,v,\xi) = \left\{  
\begin{split}
J(y,v) & + \frac{1}{2\varepsilon}\mid\mid Ay+g(y)-f-v-\xi \mid\mid_{L^{2}(\Omega)}^{2}  \\
 & + \frac{1}{2}\mid\mid A(y-y_{\alpha}) \mid\mid_{L^{2}(\Omega)}^{2}+\frac{1}{2}\mid\mid v-v_{\alpha} \mid\mid_{L^{2}(\Omega)}^{2} \\
 & + \frac{1}{2}\mid\mid \xi-\xi_{\alpha} \mid\mid_{L^{2}(\Omega)}^{2}
\end{split}
 \right.
\end{equation}

		and we consider the penalized optimization problem 
		\begin{center}
		$ ~~~~~~~~ \text{min} ~ \{ J_{\varepsilon}^{\alpha}(y,v,\xi)~\vert ~(y,v,\xi)\in \mathcal{D}_{\alpha},~ y\in H^{2}(\Omega)\cap H_{0}^{1}(\Omega)\} ~~~~~~~~~~~~~~~~~~ ~~~~~~~~~~  $\hspace{20mm} $(\mathcal{P}^{\varepsilon}_{\alpha})  $   
		\end{center}
		
	  \begin{Theorem}
	  The penalized problem $ (\mathcal{P}^{\alpha}_{\varepsilon}) $ has at least a solution $ (y_{\varepsilon},v_{\varepsilon},\xi_{\varepsilon})\in  (H^{2}(\Omega)\cap H_{0}^{1}(\Omega)) \times L^{2}(\Omega)\times L^{2}(\Omega) .$ 
	  
	   \end{Theorem}
	  
	  \begin{Proof}
	  The proof is almost the same as the one of Theoreme 3.1. The main difference is that we have no longer  $Ay_{n}+g(y_{n})-f-v_{n}-\xi_{n}=0 $, for any minimizing sequence.\\
	  Anyway, $y_{n},~ v_{n}, ~\xi_{n},~ Ay_{n}  $ and $g(y_{n})  $ are bounded in $L^{2}(\Omega)$, and it is standard to see that any weak-cluster point of this minimizing sequence is feasible and is a solution to the problem,
	  \begin{center}
	  $  Ay_{n}+g(y_{n})-f-v_{n}-\xi_{n} $ $ \rightharpoonup 0, $ weakly in $ L^{2}(\Omega) $.
	  \end{center}
	  
	  \end{Proof}
	  Now we may also give a result concerning the asymptotic behavior of the solutions of the penalized problems.

	  \begin{Theorem}
	 
	  When $   \varepsilon$ goes to $ 0 $, $ (y_{\varepsilon},v_{\varepsilon},\xi_{\varepsilon}) $ strongly convergs to $ (y_{\alpha},v_{\alpha},\xi_{\alpha})\in  (H^{2}(\Omega)\cap H_{0}^{1}(\Omega)) \times L^{2}(\Omega)\times L^{2}(\Omega) .$
	  
	  \end{Theorem}

	   \begin{Proof}
	  
	  The proof is quite similar to the one of Theoreme 3.1. We have :
		  \begin{equation} \label{eq1.21}
	  \forall \varepsilon > 0     ~~~~~~~~  J_{\varepsilon}^{\alpha}(y_{\varepsilon},v_{\varepsilon},\xi_{\varepsilon}) \leq J_{\varepsilon}^{\alpha}(y_{\alpha},v_{\alpha},\xi_{\alpha})= J(y_{\alpha},v_{\alpha})=j_{\alpha} <+\infty .
	    \end{equation}
	  
	  So 
	  \begin{center}
	 $  \frac{1}{\varepsilon}\mid\mid Ay_{\varepsilon}+g(y_{\varepsilon})-f-v_{\varepsilon}-\xi_{\varepsilon} \mid\mid_{L^{2}(\Omega)}^{2}$ +$ \mid\mid A(y_{\varepsilon}-y_{\alpha}) \mid\mid_{L^{2}(\Omega)}^{2}  $+$ \mid\mid v_{\varepsilon}-v_{\alpha} \mid\mid_{L^{2}(\Omega)}^{2}  $+$ \mid\mid \xi_{\varepsilon}-\xi_{\alpha} \mid\mid_{L^{2}(\Omega)}^{2}  $ $ \leq 2j_{\alpha}$.
	  \end{center}
	 
	Therefore $ v_{\varepsilon}, Ay_{\varepsilon} $ and $ \xi_{\varepsilon} $ are $ L^{2}(\Omega)-\text{bounded} $; this yields that $ Ay_{\varepsilon}+g(y_{\varepsilon})-f-v_{\varepsilon} $  is $ L^{2}(\Omega)\text{-bounded} $ and $y_{\varepsilon}$   is $H^{2}(\Omega)\cap H_{0}^{1}(\Omega)\text{-bounded}$. So, using Lemma 3.3, we conclude that 	
	\begin{center}
	\begin{enumerate}[label=(\roman*)]
	 \centering
	 \item$ v_{\varepsilon} $ converges to some $ \tilde{v}$ weakly in $ L^{2}(\Omega) $,\\
	
	 \item$ y_{\varepsilon} $ converges to some $ \tilde{y}$ strongly in $ H_{0}^{1}(\Omega) $,\\
	
	 \item$ \xi_{\varepsilon} $ converges to some $ \tilde{\xi}$ weakly in $ L^{2}(\Omega) $, and  \\

	  \item$Ay_{\varepsilon}+g(y_{\varepsilon})-f-v_{\varepsilon}-\xi_{\varepsilon} $ converges to $ A\tilde{y}+g(\tilde{y})-f-\tilde{v}-\tilde{\xi}   $ weakly in $L^{2}(\Omega)$. \
	 \end{enumerate}
	 \end{center}
	 
	 Moreover,  $ \mid\mid Ay_{\varepsilon}+g(y_{\varepsilon})-f-v_{\varepsilon}-\xi_{\varepsilon} \mid\mid_{L^{2}(\Omega)}^{2} ~ \leq  2\varepsilon j_{\alpha}   $ implies the strong convergence of $Ay_{\varepsilon}+g(y_{\varepsilon})-f-v_{\varepsilon}-\xi_{\varepsilon}$ to $ 0 $ in $ L^{2}(\Omega) $. Therefore  $ A\tilde{y}+g(\tilde{y})=f+\tilde{v}+\tilde{\xi}$. \\
	 It is easy to see that $\tilde{y} \in K $,   $\tilde{v} \in U_{ad}   $ and    $\tilde{\xi} \in V_{ad} $. Moreover, as  $ y_{\varepsilon} $ converges to  $\tilde{y} $ strongly in $ L^{2}(\Omega) $ and $ \xi_{\varepsilon} $ converges to $  \tilde{\xi}$ weakly in $ L^{2}(\Omega) $, we know that $ \frac{y_{\varepsilon}}{y_{\varepsilon}+\alpha}+\frac{\xi_{\varepsilon}}{\xi_{\varepsilon}+\alpha}(\leq 1) $ converges to $ \frac{\tilde{y}}{\tilde{y}+\alpha}+\frac{\tilde{\xi}}{\tilde{\xi}+\alpha} $. So $ \frac{\tilde{y}}{\tilde{y}+\alpha}+\frac{\tilde{\xi}}{\tilde{\xi}+\alpha} \leq 1 $ and $ (\tilde{y},\tilde{v},\tilde{\xi})$ belongs to $ \mathcal{D}_{\alpha} $.\\
	 Relation \eqref{eq1.21} implies that 
	 \begin{equation} \label{eq1.22}
	  J(y_{\varepsilon},v_{\varepsilon})+ \frac{1}{2}\mid\mid A(y_{\varepsilon}-y_{\alpha}) \mid\mid_{L^{2}(\Omega)}^{2}+ \frac{1}{2} \mid\mid v_{\varepsilon}-v_{\alpha} \mid\mid_{L^{2}(\Omega)}^{2}  + \frac{1}{2}\mid\mid \xi_{\varepsilon}-\xi_{\alpha} \mid\mid_{L^{2}(\Omega)}^{2}    \leq  J(y_{\alpha},v_{\alpha}).
	 \end{equation}
	 \newpage
	 Passing to the inf-limit and using the fact that $ (\tilde{y},\tilde{v},\tilde{\xi})$ belongs to $ \mathcal{D}_{\alpha}$, we obtain  
	 \begin{center}
$  J(\tilde{y},\tilde{v})+ \frac{1}{2}\mid\mid A(\tilde{y}-y_{\alpha}) \mid\mid_{L^{2}(\Omega)}^{2}+ \frac{1}{2} \mid\mid \tilde{v}-v_{\alpha} \mid\mid_{L^{2}(\Omega)}^{2}  + \frac{1}{2}\mid\mid \tilde{\xi}-\xi_{\alpha} \mid\mid_{L^{2}(\Omega)}^{2}    \leq J(y_{\alpha},v_{\alpha}) \leq J(\tilde{y},\tilde{v}).$
	 \end{center}
	  Therefore $ A(\tilde{y}-y_{\alpha})=0 $ (which implies $ \tilde{y}=y_{\alpha}$ since $ A(\tilde{y}-y_{\alpha}) \in H_{0}^{1}(\Omega))$,  $\tilde{v}=v_{\alpha}  $ and  $ \tilde{\xi}=\xi_{\alpha} $. \\
	  We just proved the weak convergence of $ (y_{\varepsilon},v_{\varepsilon},\xi_{\varepsilon}) $ to $ (y_{\alpha},v_{\alpha},\xi_{\alpha})$ in  $ H_{0}^{1}(\Omega) \times L^{2}(\Omega)\times L^{2}(\Omega), $ \\ and that $ \displaystyle{\lim_{\varepsilon \to 0}}J(y_{\varepsilon},v_{\varepsilon})=J(y_{\alpha},v_{\alpha})$. Relation \eqref{eq1.22} gives 
	  
	  \begin{center}
	    $ \mid\mid A(y_{\varepsilon}-y_{\alpha}) \mid\mid_{L^{2}(\Omega)}^{2}+  \mid\mid v_{\varepsilon}-v_{\alpha} \mid\mid_{L^{2}(\Omega)}^{2}  + \mid\mid \xi_{\varepsilon}-\xi_{\alpha} \mid\mid_{L^{2}(\Omega)}^{2} ~ \leq 2[J(y_{\alpha},v_{\alpha})-J(y_{\varepsilon},v_{\varepsilon})]; $   
	   \end{center}

	 therefore we get the strong convergence of  $ Ay_{\varepsilon} $ towards $ Ay_{\alpha} $ in $ L^{2}(\Omega) $, that is the strong convergence of $y_{\varepsilon}  $ to $ y_{\alpha} $ in $ H^{2}(\Omega) \cap H_{0}^{1}(\Omega) $. We get also the strong convergence of $ (v_{\varepsilon},\xi_{\varepsilon}) $ towards $  (v_{\alpha},\xi_{\alpha})$  in $ L^{2}(\Omega)\times L^{2}(\Omega) $. Let us remark, at last, that $ y_{\varepsilon} $ converges to $ y_{\alpha} $ uniformly in $\bar{\Omega}$, since $  H^{2}(\Omega)\cap H_{0}^{1}(\Omega)\subset {\mathcal{C}}(\bar{\Omega}).$ 
	 
	\end{Proof}
	
	\begin{Corollary}
	If we define the penalized adjoint state $ p_{\varepsilon} $ as the solution of 
	\begin{equation} \label{eq1.23}
	A^{*}p_{\varepsilon}+g^{'}(y_{\varepsilon})p_{\varepsilon}=y_{\varepsilon}-z_{d} ~~ \text{on}~~  \Omega ,  ~p_{\varepsilon}\in H^{1}_{0}(\Omega) ,
	\end{equation}
	then $ p_{\varepsilon}$ strongly converges to $  p_{\alpha}$ in $ H^{1}_{0}(\Omega) $ , where $ p_{\alpha} $ is defined by 
	\begin{equation} \label{eq1.24}
A^{*}p_{\alpha}+g^{'}(y_{\alpha})p_{\alpha}=y_{\alpha}-z_{d} ~~ \text{on} ~~  \Omega , ~ p_{\alpha}\in H^{1}_{0}(\Omega) .
	\end{equation}
		\end{Corollary}
\begin{Proof}
we have seen that $ \mid\mid y_{\varepsilon}-y_{\alpha} \mid\mid_{\infty} \to 0$. Therefore $ y_{\varepsilon} $ remains in a bounded set of $ \R^{n} $ (independent of $\varepsilon$). As $ g $ is a $ \mathcal{C}^{1} $ function, this means that $ \mid\mid g^{'}(y_{\varepsilon}) \mid\mid_{\infty}$ is bounded by a constant $ C $ which does not depend on $ \varepsilon $. In particular $g^{'}(y_{\varepsilon})$ is bounded in $L^{2}(\Omega)  $ and Lebesgue's Theorem implies the strong convergence of $g^{'}(y_{\varepsilon})$ to $g^{'}(y_{\alpha})$ in $L^{2}(\Omega)$.\\

Let $ p_{\varepsilon}  $ be the solution of \eqref{eq1.23}. This gives 
\begin{center}
    $ \langle\ A^{*}p_{\varepsilon},p_{\varepsilon} \rangle  $+$ \langle\ g^{'}(y_{\varepsilon})p_{\varepsilon},p_{\varepsilon} \rangle $=$  \langle\ y_{\varepsilon}-z_{d},p_{\varepsilon} \rangle, $

\end{center}

as $ g^{'} \geq 0 $ and $ A^{*} $ is coercive we get 
\begin{center}
$ \delta \mid\mid p_{\varepsilon} \mid\mid_{H_{0}^{1}(\Omega)}^{2}~  \leq $  $ \mid\mid y_{\varepsilon}-z_{d} \mid\mid_{H^{-1}(\Omega)} \mid\mid p_{\varepsilon} \mid\mid_{H_{0}^{1}(\Omega)}.$
\end{center}

So, $p_{\varepsilon}$ is bounded in  $ H_{0}^{1}(\Omega) $ and weakly converges to $ \tilde{p} $ in $ H_{0}^{1}(\Omega) $ . Moreover, $ p_{\varepsilon}$ is the solution to 
\begin{center}
$A^{*}p_{\varepsilon}= -g^{'}(y_{\varepsilon})p_{\varepsilon}+y_{\varepsilon}-z_{d}$ on $ \Omega, $
\end{center}
the left-hande side (weakly) converges to $-g^{'}(y_{\alpha}) \tilde{p}+y_{\alpha}-z_{d}$ in $L^{2}(\Omega)  $; this achieves the proof.

\end{Proof}

	\subsection{Optimality conditions for the penalized problem}
	
We apply Theorem 4.1 to the above penalized problem $ (\mathcal{P}_{\alpha}^{\varepsilon}) $. We set 

\begin{center}
$  x=y,~~  $ $ u=(v,\xi),~~$ $(x_{0},u_{0})=(x_{\varepsilon},v_{\varepsilon},\xi_{\varepsilon})$ \\
\vspace{0.2cm}
$ \mathcal{X}= H^{2}(\Omega) \cap H^{1}_{0}(\Omega),~~ $ $  \mathcal{Z}_{2}= \mathcal{X}$\\  
$  \mathcal{U}= L^{2}(\Omega)\times L^{2}(\Omega)  $ \\
\vspace{0.2cm}
$ \mathcal{U}_{ad}=U_{ad} \times V_{ad} ,~~ $ 
$ P=\{y \in H^{2}(\Omega) \cap H^{1}_{0}(\Omega) ~ \vert ~ y \geq 0 \}\times \R^{+}$ \\
\end{center}
We recall that $\langle\ , \rangle$  denote the $  L^{2}(\Omega)$-scalar product, and  
\begin{center}
$ G(y,v,\xi)=(-y,\langle\ 1, \frac{y}{y+\alpha}+\frac{\xi}{\xi+\alpha} \rangle -Area(\Omega)),~~  $ 
\vspace{0.2cm}
$ f(x,u)=J_{\alpha}^{\varepsilon}(y,v,\xi).$
\end{center}
\newpage
There is no equality constraint and $ G $ is $ \mathcal{C}^{1},$ 
\begin{center}
$ G^{'}(y_{\varepsilon},v_{\epsilon},\xi_{\epsilon})(y,v,\xi) $ =  $ (-y, \langle\ y, \frac{\alpha}{(y_{\varepsilon}+\alpha)^{2}} \rangle+ \langle\ \xi, \frac{\alpha}{(\xi_{\varepsilon}+\alpha)^{2}} \rangle)$.
\end{center}
Here
\begin{center}
 $ \mathcal{U}_{ad}(v_{\varepsilon},\xi_{\varepsilon}) $= $\{(\lambda(v-v_{\varepsilon}),\mu(\xi-\xi_{\varepsilon}))~ \vert ~\lambda \geq 0,~ \mu \geq 0,~ v\in U_{ad}, ~\xi \in V_{ad}\},$ \\
 \vspace{0.3cm}
 $ P (G(y_{\varepsilon},v_{\varepsilon},\xi_{\varepsilon})) $=$ \{(-p+\lambda y_{\varepsilon}, -\gamma-\lambda(\langle\ 1, \frac{y_{\varepsilon}}{y_{\varepsilon}+\alpha}+\frac{\xi_{\varepsilon}}{\xi_{\varepsilon}+\alpha} \rangle -Area(\Omega)) \in H^{2}(\Omega) \cap H_{0}^{1}(\Omega) \times \R ~\vert ~ \gamma, ~\lambda \geq 0, ~p \geq 0        \} $
 \end{center} 
 Let us write the the condition \eqref{eq1.16} : for any $ (z,\beta) $ in $ \mathcal{X}\times \R $ we must solve the system : 
  \begin{center}
   $\begin{aligned}
		  -y+p-\lambda y_{\varepsilon}~~~~~~~~~~~~~~~~~~~~~~~~~~~~~~~~~~~~~~~~~~~~~~~~~~~~~~~~~~~~~~~~&=~&z ,\\  
		  \langle\ y, \frac{\alpha}{(y_{\varepsilon}+\alpha)^{2}} \rangle   +  \langle\ \mu (\xi-\xi_{\varepsilon}), \frac{\alpha}{(\xi_{\varepsilon}+\alpha)^{2}} \rangle  +\gamma +\lambda(\langle\ 1, \frac{y_{\varepsilon}}{y_{\varepsilon}+\alpha}+\frac{\xi_{\varepsilon}}{\xi_{\varepsilon}+\alpha} \rangle -Area(\Omega))&=~&\beta, \\
		\end{aligned}$
  \end{center}
with $  \mu, ~\gamma,~ \lambda \geq 0,~ $  $ \xi \in V_{ad},~ v\in U_{ad} $, and $  y\in \mathcal{X}$. Taking $ y $ from the first equation into the second we have to solve : 
\begin{center}
  $\langle\ p-\lambda y_{\varepsilon}-z, \frac{\alpha}{(y_{\varepsilon}+\alpha)^{2}} \rangle  $ +$  \langle\ \mu (\xi-\xi_{\varepsilon}), \frac{\alpha}{(\xi_{\varepsilon}+\alpha)^{2}} \rangle  +\gamma +\lambda(\langle\ 1, \frac{y_{\varepsilon}}{y_{\varepsilon}+\alpha}+\frac{\xi_{\varepsilon}}{\xi_{\varepsilon}+\alpha} \rangle -Area(\Omega))= \beta. $
\end{center}

\vspace{1cm}
So 
\begin{center}
 $\langle\ p, \frac{\alpha}{(y_{\varepsilon}+\alpha)^{2}} \rangle  $ -$ \lambda \langle\  y_{\varepsilon}, \frac{\alpha}{(y_{\varepsilon}+\alpha)^{2}} \rangle  $ +$  \langle\ \mu (\xi-\xi_{\varepsilon}), \frac{\alpha}{(\xi_{\varepsilon}+\alpha)^{2}} \rangle  +\gamma +\lambda(\langle\ 1, \frac{y_{\varepsilon}}{y_{\varepsilon}+\alpha}+\frac{\xi_{\varepsilon}}{\xi_{\varepsilon}+\alpha} \rangle -Area(\Omega))= \beta $ + $\langle\ z, \frac{\alpha}{(y_{\varepsilon}+\alpha)^{2}} \rangle = \rho $
\end{center}
with $  \mu,~ \gamma, ~\lambda \geq 0,~$  $ \xi \in V_{ad},~ v\in U_{ad} $. We see that we may take : $  \mu=1,~ \xi=\xi_{\varepsilon},~ p=0,$ and 
\begin{itemize}
  \item If $\rho \geq 0  $, we choose $  \lambda=0, ~$ $   \gamma=\rho$
  \item If $ \rho <0 $, we have two cases :\\
\end{itemize}
 $ ~~~~~~~~~~$   - If  $(\langle\ 1, \frac{y_{\varepsilon}}{y_{\varepsilon}+\alpha}+\frac{\xi_{\varepsilon}}{\xi_{\varepsilon}+\alpha} \rangle -Area(\Omega)) = \zeta <0 $, then we set $ \gamma=  \lambda \langle\  y_{\varepsilon}, \frac{\alpha}{(y_{\varepsilon}+\alpha)^{2}} \rangle , ~$ $ \lambda=\frac{\rho}{\zeta}. $ \\
\vspace{0.3cm}
  $ ~~~~~~~~~~$  - If $(\langle\ 1, \frac{y_{\varepsilon}}{y_{\varepsilon}+\alpha}+\frac{\xi_{\varepsilon}}{\xi_{\varepsilon}+\alpha} \rangle -Area(\Omega)) =0$, then we set $  \gamma=0,~ \lambda = -\frac{\rho}{\eta}$, such that $ \eta= \lambda \langle\  y_{\varepsilon}, \frac{\alpha}{(y_{\varepsilon}+\alpha)^{2}} \rangle.  $ \\
 Indeed, we have 
 \begin{center}
  $(\langle\ 1, \frac{y_{\varepsilon}}{y_{\varepsilon}+\alpha}+\frac{\xi_{\varepsilon}}{\xi_{\varepsilon}+\alpha} \rangle -Area(\Omega)) =0,$
 \end{center}
 in view of Lemma 3.2, we have 
 \begin{center}
    $(\langle\ 1, \frac{y_{\varepsilon}}{y_{\varepsilon}+\alpha}+\frac{\xi_{\varepsilon}}{\xi_{\varepsilon}+\alpha} \rangle -Area(\Omega)) =0 \Longleftrightarrow y_{\varepsilon}.\xi_{\varepsilon}=\alpha^{2}  ~ \text{a.e} ~\text{in} ~ \Omega.$
    \end{center}   
  Therefore $ y_{\varepsilon} $ and $\xi_{\varepsilon}$ are strictly positive.  
   (Here $ y_{\varepsilon} >0,~\text{and} ~\xi_{\varepsilon} >0,  ~~~ \text{becauce} ~~ \alpha >0 ~~$ fixed). Hence, $ \eta >0 $ and $ \lambda > 0. $  \\
\vspace{0.4cm}
So condition \eqref{eq1.16} is always satisfied and we may apply Theorem 4.1, since $ J_{\varepsilon}^{\alpha} $ is Fréchet differentiable, and
\begin{center}
$J_{\varepsilon}^{\alpha^{'}}(y_{\varepsilon},v_{\varepsilon},\xi_{\varepsilon})(y,v,\xi) =$ $ \left( \begin{array}{lcrr} (J_{\varepsilon}^{\alpha})_{y}^{'}(y_{\varepsilon},v_{\varepsilon},\xi_{\varepsilon}) & (J_{\varepsilon}^{\alpha})_{v}^{'}(y_{\varepsilon},v_{\varepsilon},\xi_{\varepsilon}) & (J_{\varepsilon}^{\alpha})_{\xi}^{'}(y_{\varepsilon},v_{\varepsilon},\xi_{\varepsilon}) \\
 \end{array} \right). \left( \begin{array}{ccc}
y  \\
v  \\
\xi  \end{array} \right).$
\end{center}
We have : 
		\begin{equation} 
		J_{\varepsilon}^{\alpha}(y,v,\xi) = \left\{  
\begin{split}
J(y,v) & + \frac{1}{2\varepsilon}\mid\mid Ay+g(y)-f-v-\xi \mid\mid_{L^{2}(\Omega)}^{2}  \\
 & + \frac{1}{2}\mid\mid A(y-y_{\alpha}) \mid\mid_{L^{2}(\Omega)}^{2}+\frac{1}{2}\mid\mid v-v_{\alpha} \mid\mid_{L^{2}(\Omega)}^{2} \\
 & + \frac{1}{2}\mid\mid \xi-\xi_{\alpha} \mid\mid_{L^{2}(\Omega)}^{2}.
\end{split}
 \right.
\end{equation}
\newpage
		So,
		\begin{center}
		$ (J_{\varepsilon}^{\alpha})_{y}^{'}(y_{\varepsilon},v_{\varepsilon},\xi_{\varepsilon}) $= \(\langle\ 1,y_{\varepsilon}-z_{d} \rangle + \frac{1}{\varepsilon} \langle\ A+g^{'}(y_{\varepsilon}), Ay_{\varepsilon}+g(y_{\varepsilon})-f-v_{\varepsilon}-\xi_{\varepsilon} \rangle +\langle\ A, A(y_{\varepsilon}-y_{\alpha})  \rangle \). \\
		$ (J_{\varepsilon}^{\alpha})_{v}^{'}(y_{\varepsilon},v_{\varepsilon},\xi_{\varepsilon}) $= \(\langle\ \nu ,v_{\varepsilon}-v_{d} \rangle +\langle\ 1,v_{\varepsilon}-v_{\alpha} \rangle - \frac{1}{\varepsilon} \langle\ 1, Ay_{\varepsilon}+g(y_{\varepsilon})-f-v_{\varepsilon}-\xi_{\varepsilon} \rangle  \). \\
		$ (J_{\varepsilon}^{\alpha})_{\xi}^{'}(y_{\varepsilon},v_{\varepsilon},\xi_{\varepsilon}) $= \(\langle\ 1 ,\xi_{\varepsilon}-\xi_{\alpha} \rangle  - \frac{1}{\varepsilon} \langle\ 1, Ay_{\varepsilon}+g(y_{\varepsilon})-f-v_{\varepsilon}-\xi_{\varepsilon} \rangle  \). \\
		\end{center}
		Therefore \\
	$  J_{\varepsilon}^{\alpha^{'}}(y_{\varepsilon},v_{\varepsilon},\xi_{\varepsilon})(y,v,\xi)= \langle\ y ,y_{\varepsilon}-z_{d} \rangle +\nu \langle\  v ,v_{\varepsilon}-v_{d} \rangle + \langle\ v,v_{\varepsilon}-v_{\alpha} \rangle + \langle\ \xi ,\xi_{\varepsilon}-\xi_{\alpha} \rangle +\langle\ Ay, A(y_{\varepsilon}-y_{\alpha})  \rangle +\langle\ q_{\varepsilon}, A_{\varepsilon}y-v-\xi  \rangle $, \\
	\vspace{0.3cm}
	where 
	\begin{equation} \label{eq1.25}
	  q_{\varepsilon}=\dfrac{Ay_{\varepsilon}+g(y_{\varepsilon})-f-v_{\varepsilon}-\xi_{\varepsilon}}{\varepsilon} \in L^{2}(\Omega) ~~~~ \text{and}~~~~  A_{\varepsilon}=A+g^{'}(y_{\varepsilon}). 
	\end{equation}
	
Thers exists $ s_{\varepsilon} \in \mathcal{X}^{*} $ and $ r_{\varepsilon} \in \R $ such that : 
\begin{equation} \label{eq1.26}
 \forall y \in \mathcal{X} \quad    \langle\ y ,y_{\varepsilon}-z_{d} \rangle +\langle\ q_{\varepsilon} , A_{\varepsilon}y \rangle + \langle\ Ay , A(y_{\varepsilon}-y_{\alpha}) \rangle +r_{\varepsilon} \langle\ y , \frac{\alpha}{(y_{\varepsilon}+\alpha)^{2}} \rangle -\langle\ \langle\ s_{\varepsilon} , y \rangle \rangle = 0,   
\end{equation}
\begin{equation} \label{eq1.27}
 \forall v\in U_{ad}  \quad     \langle\ \nu (v_{\varepsilon}-v_{d})+v_{\varepsilon}-v_{\alpha}-q_{\varepsilon}, v-v_{\varepsilon} \rangle \geq 0,
\end{equation}
 \begin{equation} \label{eq1.28}
\forall \xi \in V_{ad}  \quad   \langle\ r_{\varepsilon}\frac{\alpha}{(\xi_{\alpha}+\alpha)^{2}}-q_{\varepsilon}+\xi_{\varepsilon}-\xi_{\alpha}, \xi-\xi_{\varepsilon} \rangle \geq 0,
\end{equation}
\begin{equation} \label{eq1.29}
 r_{\varepsilon}\geq 0,\quad    r_{\varepsilon}( \langle\ 1, \frac{y_{\varepsilon}}{y_{\varepsilon}+\alpha}+\frac{\xi_{\varepsilon}}{\xi_{\varepsilon}+\alpha} \rangle -Area(\Omega) )=0,
\end{equation}
\begin{equation} \label{eq1.30}
 \forall y\in \mathcal{X}, ~ y\geq 0,  \quad  \langle\ \langle\ s_{\varepsilon} , y \rangle \rangle \geq 0, ~~ \langle\ \langle\ s_{\varepsilon} , y_{\varepsilon} \rangle \rangle =0,
\end{equation}
	where $ \langle\ \langle\ ,\rangle\rangle$ denotes the duality product between $ \mathcal{X}^{*} $ and $ \mathcal{X} $. \\
	Finally, we have optimality conditions on the penalized system, without any further assumption  : 
	\begin{Theorem} 
	The solution $ (y_{\varepsilon},v_{\varepsilon},\xi_{\varepsilon})$  of problem $ (\mathcal{P}_{\varepsilon}^{\alpha}) $ satisfies the following optimality  system :
\begin{equation} \label{eq1.31}
 \forall y\in \tilde{K}   \quad  \langle\ p_{\varepsilon}+q_{\varepsilon} ,A_{\varepsilon}(y-y_{\varepsilon}) \rangle + \langle\ A(y-y_{\varepsilon}) , A(y_{\varepsilon}-y_{\alpha}) \rangle +r_{\varepsilon} \langle\ y-y_{\varepsilon} , \frac{\alpha}{(y_{\varepsilon}+\alpha)^{2}} \rangle \geq 0, 
\end{equation}
\begin{equation} \label{eq1.32}
 \forall v\in U_{ad}  \quad    \langle\ \nu (v_{\varepsilon}-v_{d})+v_{\varepsilon}-v_{\alpha}-q_{\varepsilon}, v-v_{\varepsilon} \rangle \geq 0,
\end{equation}
 \begin{equation} \label{eq1.33}
 \forall \xi \in V_{ad}  \quad   \langle\ r_{\varepsilon}\frac{\alpha}{(\xi_{\alpha}+\alpha)^{2}}-q_{\varepsilon}+\xi_{\varepsilon}-\xi_{\alpha}, \xi-\xi_{\varepsilon} \rangle \geq 0,
\end{equation}
\begin{equation} \label{eq1.34}
 r_{\varepsilon}\geq 0, ~~  r_{\varepsilon} \biggl( \langle\ 1, \frac{y_{\varepsilon}}{y_{\varepsilon}+\alpha}+\frac{\xi_{\varepsilon}}{\xi_{\varepsilon}+\alpha} \rangle -Area(\Omega) \biggr) =0,
\end{equation}
	where $ p_{\varepsilon} $ is given by \eqref{eq1.23} and $ q_{\varepsilon} $ by \eqref{eq1.25}. 
	\end{Theorem}
	\begin{Proof}
	Relation \eqref{eq1.26} applied to $ y-y_{\varepsilon} $ gives : 
	\begin{center}
	$ \forall y\in \mathcal{X}  \quad  \langle\ y-y_{\varepsilon} ,y_{\varepsilon}-z_{d} \rangle +\langle\ q_{\varepsilon} , A_{\varepsilon}(y-y_{\varepsilon} \rangle + \langle\ A(y-y_{\varepsilon}) , A(y_{\varepsilon}-y_{\alpha}) \rangle +r_{\epsilon} \langle\ y-y_{\varepsilon} , \frac{\alpha}{(y_{\varepsilon}+\alpha)^{2}} \rangle =\langle\ \langle\ s_{\varepsilon} , y \rangle \rangle - \langle\ \langle\ s_{\varepsilon} , y_{\varepsilon} \rangle \rangle ,$            
	\end{center}
	So, with \eqref{eq1.30}, we obtain 
	\begin{center}
	$ \forall y\in \tilde{K}  $   $ \quad $  $ \langle\ p_{\varepsilon}+q_{\varepsilon} ,A_{\varepsilon}(y-y_{\varepsilon} \rangle + \langle\ A(y-y_{\varepsilon}) , A(y_{\varepsilon}-y_{\alpha}) \rangle +r_{\varepsilon} \langle\ y-y_{\varepsilon} , \frac{\alpha}{(y_{\varepsilon}+\alpha)^{2}} \rangle \geq 0 ,$ \\ 
	\end{center}
where $ p_{\varepsilon} $ is given by \eqref{eq1.23} and $ q_{\varepsilon} $ by \eqref{eq1.25}, and $\tilde{K} = K \cap (H^{2}(\Omega) \cap H_{0}^{1}(\Omega)). $   
	\end{Proof}
	
	\section{OPTIMALITY CONDITIONS FOR $( \mathcal{P}^{\alpha}) $}

	\subsection{Qualification assumption }
	
	Now we would like to study the asymptotic behaviour of the previous optimality conditions \eqref{eq1.31}-\eqref{eq1.34} when $  \varepsilon $ goes to $ 0 $ and we need some estimations on $ q_{\varepsilon} $ and $ r_{\varepsilon}$. We have to assume some qualification conditions to pass to the limit in the penalized optimality system, we remark that 
	
	\begin{center}
	$ A_{\varepsilon}y_{\varepsilon}-v_{\varepsilon}-\xi_{\varepsilon}=Ay_{\varepsilon}+g(y_{\varepsilon})-v_{\varepsilon}-\xi_{\varepsilon}-f+f+g^{'}(y_{\varepsilon})y_{\varepsilon}-g(y_{\varepsilon}) . $
	\end{center}
	We set 
	\begin{equation} \label{eq1.35}	
	   \omega_{\varepsilon}=g^{'}(y_{\varepsilon})y_{\varepsilon}-g(y_{\varepsilon}) ~~~\text{and}~~~   \omega_{\alpha}=g^{'}(y_{\alpha})y_{\alpha}-g(y_{\alpha}),
	\end{equation}
	
	 so that 
	 \begin{center}
	 $ A_{\varepsilon}y_{\varepsilon}-v_{\varepsilon}-\xi_{\varepsilon}= \varepsilon q_{\varepsilon}+f+\omega_{\varepsilon}.$
	 \end{center}
Let us choose $ (y,v,\xi) $ in $ 	\tilde{K} \times U_{ad} \times V_{ad},$ and add relation \eqref{eq1.31}-\eqref{eq1.33}. We have : 
\\
\begin{center}
$  \langle\ p_{\varepsilon}, A_{\varepsilon}(y-y_{\varepsilon}) \rangle+
 \langle\ q_{\varepsilon},A_{\epsilon}(y-y_{\varepsilon}) \rangle + \langle\ A(y-y_{\varepsilon}) , A(y_{\varepsilon}-y_{\alpha}) \rangle +
 \langle\ \nu(v_{\varepsilon}-v_{d})+v_{\varepsilon}-v_{\alpha}, v-v_{\varepsilon} \rangle $  \\ 
	\vspace{0.2cm}
	$  
	+ \langle\ -q_{\varepsilon}, v-v_{\varepsilon} \rangle+ 
r_{\varepsilon} \langle\ y-y_{\epsilon} , \frac{\alpha}{(y_{\varepsilon}+\alpha)^{2}} \rangle     +r_{\varepsilon} \langle\ \frac{\alpha}{(\xi_{\alpha}+\alpha)^{2}}, \xi-\xi_{\varepsilon} \rangle +\langle\ \xi_{\varepsilon}-\xi_{\alpha}, \xi-\xi_{\varepsilon} \rangle + \langle\ -q_{\varepsilon}, \xi-\xi_{\varepsilon} \rangle \geq 0. $
\end{center}
	\vspace{0.5cm}
	So that: 
	\begin{center}
	$ \langle\ q_{\varepsilon},f+\omega_{\varepsilon}+v+\xi-A_{\varepsilon}y \rangle -r_{\varepsilon} \biggl( \langle\ \frac{\alpha}{(y_{\varepsilon}+\alpha)^{2}} ,y-y_{\varepsilon}  \rangle + \langle\ \frac{\alpha}{(\xi_{\alpha}+\alpha)^{2}}, \xi-\xi_{\varepsilon} \rangle \biggr) $ $  \leq $ \\ 
	\vspace{0.2cm}
	$
	\langle\ p_{\varepsilon},A_{\varepsilon}(y-y_{\varepsilon} \rangle + \langle\ A(y-y_{\varepsilon}) , A(y_{\varepsilon}-y_{\alpha}) \rangle + \langle\ \nu (v_{\varepsilon}-v_{d})+v_{\varepsilon}-v_{\alpha}, v-v_{\varepsilon} \rangle +\langle\ \xi_{\varepsilon}-\xi_{\alpha}, \xi-\xi_{\varepsilon} \rangle -\varepsilon \mid\mid q _{\varepsilon}\mid\mid_{2}^{2}.  $   
	\vspace{0.2cm}
	\end{center}

	The right hand side is uniformly bounded with respect to $\varepsilon  $ by a constant $ C $ which only depends of  $ y, ~$ $ v,~ $ $ \xi $. Here we use as well Theorem 5.2.
	Moreover relation \eqref{eq1.34} gives 
	
	\begin{center}
	 $ r_{\varepsilon}  \langle\ 1, \frac{y_{\varepsilon}}{y_{\varepsilon}+\alpha}+\frac{\xi_{\varepsilon}}{\xi_{\varepsilon}+\alpha} \rangle  =r_{\varepsilon} Area(\Omega),$
	\end{center}
	so that we finally obtain : 
	
	\begin{equation} \label{eq1.36}	
	 -\langle\ q_{\varepsilon},Ay+g^{'}(y_{\varepsilon})y-f-v-\xi-\omega_{\varepsilon} \rangle -r_{\varepsilon} \biggl(  \langle\ \frac{\alpha}{(y_{\varepsilon}+\alpha)^{2}} ,y-y_{\varepsilon}  \rangle + \langle\ \frac{\alpha}{(\xi_{\varepsilon}+\alpha)^{2}}, \xi-\xi_{\varepsilon} \rangle \biggr)  \leq ~ C_{(y, v, \xi)},
	\end{equation}
	where 
	\begin{center}
	$  q_{\varepsilon}=\dfrac{Ay_{\varepsilon}+g(y_{\varepsilon})-f-v_{\varepsilon}-\xi_{\varepsilon}}{\varepsilon} \in L^{2}(\Omega)  ~$ and $~ A_{\varepsilon}=A+g^{'}(y_{\varepsilon}) $, $~ \omega_{\varepsilon}=g^{'}(y_{\varepsilon})y_{\varepsilon}-g(y_{\varepsilon}). $ 
	\end{center}
	We consider two cases :
	\begin{enumerate}[label=(\roman*)]
	
	 \item If 
	\begin{center}
	 $  \langle\ 1, \frac{y_{\alpha}}{y_{\alpha}+\alpha}+\frac{\xi_{\alpha}}{\xi_{\alpha}+\alpha} \rangle < Area(\Omega) ,$ 
	\end{center}
	as $ \langle\ 1, \frac{y_{\varepsilon}}{y_{\varepsilon}+\alpha}+\frac{\xi_{\varepsilon}}{\xi_{\varepsilon}+\alpha} \rangle \to \langle\ 1, \frac{y_{\alpha}}{y_{\alpha}+\alpha}+\frac{\xi_{\alpha}}{\xi_{\alpha}+\alpha} \rangle $, there exists $ \varepsilon_{0} > 0$ such that 
	\begin{center}
	$ \forall \varepsilon ~\leq \varepsilon_{0} ~~~~  $ $ \langle\ 1, \frac{y_{\varepsilon}}{y_{\varepsilon}+\alpha}+\frac{\xi_{\varepsilon}}{\xi_{\varepsilon}+\alpha} \rangle < Area(\Omega), $
	  \end{center}
	  and relation \eqref{eq1.34} implies that $ r_{\varepsilon} =0$. So the limit value is $ r_{\alpha}=0 $.
	   \item If 
	  \begin{center}
	   $  \langle\ 1, \frac{y_{\alpha}}{y_{\alpha}+\alpha}+\frac{\xi_{\alpha}}{\xi_{\alpha}+\alpha} \rangle = Area(\Omega),$
	  \end{center}
	  as $ \langle\ 1, \frac{y_{\varepsilon}}{y_{\varepsilon}+\alpha}+\frac{\xi_{\varepsilon}}{\xi_{\varepsilon}+\alpha} \rangle \to \langle\ 1, \frac{y_{\alpha}}{y_{\alpha}+\alpha}+\frac{\xi_{\alpha}}{\xi_{\alpha}+\alpha} \rangle $, there exists $ \varepsilon_{0} > 0$ such that 
	\begin{center}
	$ \forall \varepsilon ~ \leq \varepsilon_{0} ~~~ $ $ \langle\ 1, \frac{y_{\varepsilon}}{y_{\varepsilon}+\alpha}+\frac{\xi_{\varepsilon}}{\xi_{\varepsilon}+\alpha} \rangle =Area(\Omega), $
	  \end{center}
	   \end{enumerate}
	
	  we cannot conclude immediately, so we assume the following condition : 
	 \begin{center}
	  $ \forall \alpha >0 $ such that $ \langle\ 1, \frac{y_{\varepsilon}}{y_{\epsilon}+\alpha}+\frac{\xi_{\varepsilon}}{\xi_{\varepsilon}+\alpha} \rangle =Area(\Omega),$ \\
	    \vspace{0.2cm}
	 $~~~~~~~~~~~~~~~~~~~~~~~~~~~~~~~ g^{'} $ is locally lipschitz continuous,  \hspace{40mm} $(\mathcal{H}_{1})  $ \V \\
	   \vspace{0.2cm}
	 
	 $ U_{ad} $ has a non empty $ L^{\infty}\text{-interior} $ (denoted $ \text{Int}_{\infty}(U_{ad})$) and that $  -(f+\omega_{\alpha})\in \text{Int}_{\infty}(U_{ad}).$ 
	
	 \end{center}

	\begin{Theorem} 
	Assume ($ \mathcal{H}_{1} $), then $ r_{\varepsilon} $ is bounded by a constant independent of $ \varepsilon $ and we may extract a subsequence that converges to $ r_{\alpha} $.

	\end{Theorem}

	  \begin{Proof}
	
	We have already mentioned that $ r_{\alpha}=0 $ when  $ \langle\ 1, \frac{y_{\varepsilon}}{y_{\varepsilon}+\alpha}+\frac{\xi_{\varepsilon}}{\xi_{\varepsilon}+\alpha} \rangle < Area(\Omega).$ In the other case, as $ g^{'} $ is locally lipschitz continuous, then $ \omega_{\varepsilon} $ uniformly converges to $ \omega_{\alpha} $ on $\bar{\Omega}  $. \\
	Indeed, we have proved that $ y_{\varepsilon} $ uniformly converges to $ y_{\alpha} $. Therefore, there exists $ \varepsilon_{0} > 0$ such that $ y_{\varepsilon}-y_{\alpha} $ remains in a bounded subset of $ \R^{n} $ independently of $ \varepsilon \in ]0, \varepsilon_{0}[$. The local lipschitz continuity of $ g^{'} $ yields 
	\begin{center}
	  $  \vert g^{'}(y_{\varepsilon}(x))-g^{'}(y_{\alpha}(x)) \vert  \leq  M\vert y_{\varepsilon}(x)-y_{\alpha}(x)\vert \leq M\mid\mid y_{\varepsilon}-y_{\alpha}\mid\mid_{\infty},~~  \forall x \in \Omega     $ 
	  \end{center}  
	
	where $ M $ is a constant that does not depend of $\varepsilon  $. Thus $ \mid\mid g^{'}(y_{\varepsilon})-g^{'}(y_{\alpha})\mid\mid_{\infty}  \to 0$. As 
	\begin{center}
	 $\vert  g^{'}(y_{\varepsilon})y_{\varepsilon}-g^{'}(y_{\alpha})y_{\alpha} \vert \leq \vert g^{'}(y_{\varepsilon}) \vert \vert y_{\varepsilon}-y_{\alpha} \vert + \vert  g^{'}(y_{\varepsilon})-g^{'}(y_{\alpha})\vert \vert y_{\alpha}\vert, $
	\end{center}
	we get 
	\begin{center}
	$  \mid\mid g^{'}(y_{\varepsilon})y_{\varepsilon}-g^{'}(y_{\alpha})y_{\alpha}\mid\mid_{\infty} \leq M\mid\mid y_{\varepsilon}-y_{\alpha}\mid\mid_{\infty}+\mid\mid g^{'}(y_{\epsilon})-g^{'}(y_{\alpha})\mid\mid_{\infty}\mid\mid y_{\alpha}\mid\mid_{\infty} \to 0.  $
	\end{center}
	Similarly $\mid\mid g(y_{\varepsilon})-g(y_{\alpha}) \mid\mid_{\infty} \to 0$. As we supposed $  -(f+\omega_{\alpha})\in \text{Int}_{\infty}(U_{ad})$, then $  -(f+\omega_{\varepsilon})\in U_{ad}$ for  $\varepsilon$ smaller than some 
	$\varepsilon_{0} > 0 $. \\ 
	
	Now, we choose $  y=0, v=-(f+\omega_{\varepsilon})$  and $ \xi=0 $ in relation \eqref{eq1.36}. We obtain 
	\begin{center}
	$ \forall \varepsilon \leq \varepsilon_{0}  ~~~~~  r_{\varepsilon} \biggl(  \langle\ \frac{\alpha}{(y_{\varepsilon}+\alpha)^{2}} ,y_{\varepsilon}  \rangle + \langle\ \frac{\alpha}{(\xi_{\varepsilon}+\alpha)^{2}}, \xi_{\varepsilon} \rangle \biggr) \leq C, $
	\end{center}
	where $ C $ is independent of $ \varepsilon $ since $ \omega_{\varepsilon} $ is uniformly bounded with respect to $\varepsilon $, for $ \varepsilon \in ]0, \varepsilon_{0}[.$ \\

	\vspace{0.4cm}
	We still need  to proof that :
	\begin{center}
	$  \langle\ \frac{\alpha}{(y_{\varepsilon}+\alpha)^{2}} ,y_{\varepsilon}  \rangle + \langle\ \frac{\alpha}{(\xi_{\varepsilon}+\alpha)^{2}}, \xi_{\varepsilon} \rangle \neq 0,~~é  \forall \alpha >0 $ and $ \varepsilon \to 0 .$
	\end{center}
	As 
	\begin{center}
	$  \langle\ 1, \frac{y_{\varepsilon}}{y_{\varepsilon}+\alpha}+\frac{\xi_{\varepsilon}}{\xi_{\varepsilon}+\alpha} \rangle =Area(\Omega),$  
	\end{center}
	So we have :
	\begin{center}
	$  \frac{y_{\varepsilon}}{y_{\varepsilon}+\alpha}+\frac{\xi_{\varepsilon}}{\xi_{\varepsilon}+\alpha} =1, $  $ ~\text{a.e.} ~ \text{in} ~ \Omega, $
	\end{center}
	in view of section 3 we obtain 
	\begin{center}
	$  \frac{y_{\varepsilon}}{y_{\varepsilon}+\alpha}+\frac{\xi_{\varepsilon}}{\xi_{\varepsilon}+\alpha} =1$ $  \Longleftrightarrow y_{\varepsilon}\xi_{\varepsilon}=\alpha^{2}\Longrightarrow  \langle\ y_{\varepsilon},\xi_{\varepsilon} \rangle = \alpha^{2}Area(\Omega)   $  $ ~~ \text{a.e.} ~~\text{in} ~~   \Omega$.
	\end{center}
	\vspace{0.2cm}
	Therefore, the set $ \{x\in \Omega , / y_{\varepsilon}(x)\neq 0 ,~ \xi_{\varepsilon}(x)\neq 0\} $ is not empty,  and the set $ \{x\in \Omega , / y_{\varepsilon}(x)= 0,~ \xi_{\varepsilon}(x)= 0\} $ is empty when $ \varepsilon $ goes to $ 0 $, since $ \alpha $ is fixed. Hence we obtain : 
	\begin{center}
	$  \langle\ \frac{\alpha}{(y_{\varepsilon}+\alpha)^{2}} ,y_{\varepsilon}  \rangle + \langle\ \frac{\alpha}{(\xi_{\varepsilon}+\alpha)^{2}}, \xi_{\varepsilon} \rangle \neq 0, \forall \alpha >0. $ 
	\end{center}

	\vspace{0.4cm}
	Finally, the passage to limit as $ \varepsilon \to 0 $ gives : 
	\begin{center}
	$  \langle\ \frac{\alpha}{(y_{\alpha}+\alpha)^{2}} ,y_{\alpha}  \rangle + \langle\ \frac{\alpha}{(\xi_{\alpha}+\alpha)^{2}}, \xi_{\alpha} \rangle \neq 0, \forall \alpha >0. $
\end{center}
	
	\end{Proof}

	Once we have the previous estimate, relation \eqref{eq1.36} becomes : 
	\begin{equation} \label{eq1.37}		
	\forall(y,v,\xi) \in \tilde{K}\times U_{ad}\times V_{ad}   ~~~~   -\langle\ q_{\varepsilon},Ay+g^{'}(y_{\varepsilon})y-f-v-\xi-\omega_{\varepsilon} \rangle  ~ \leq ~C_{(y, v, \xi)}.  
	\end{equation}
	
	Then we have to do another assumption to get the estimation of $ q_{\varepsilon}$ : 
	\begin{center}
	$ \exists p \in [1, +\infty], ~~  \exists \varepsilon_{0}>0, ~~ \exists \rho>0$ \\
	  \vspace{0.2cm}
$ \forall \varepsilon \in ]0, \varepsilon_{0}[,~~  $  $\forall \chi \in L^{p}(\Omega)$ such that $\mid\mid \chi\mid\mid_{L^{p}(\Omega)} \leq 1$, \\ 
  \vspace{0.2cm}
$ \exists (y_{\chi}^{\varepsilon},v_{\chi}^{\varepsilon},\xi_{\chi}^{\varepsilon}) $ bounded in $ \tilde{K} \times U_{ad} \times V_{ad} $ (uniformly with respect to $ \chi $ and $\varepsilon $),$ ~~~ $ \hspace{15mm} $(\mathcal{H}_{2})  $ \V\\
  \vspace{0.2cm}
such that $ Ay_{\chi}^{\varepsilon}+g^{'}(y_{\varepsilon})y_{\chi}^{\varepsilon}=f+\omega_{\varepsilon}+v_{\chi}^{\varepsilon}+\xi_{\chi}^{\varepsilon}-\rho\chi $ in $ \Omega $. 
	\end{center}
	Then we may conclude : \\

	\begin{Theorem} 
	Assume ($ \mathcal{H}_{1}) $ and ($ \mathcal{H}_{2} $), then $ q_{\varepsilon} $ is bounded in   $L^{p^{'}}(\Omega)$ by a constant independent of  $ \varepsilon $ \\ $( \text{here} \quad \frac{1}{p}+\frac{1}{p^{'}} =1)$.

	\end{Theorem}
	
	\begin{Proof}
($ \mathcal{H}_{2} $) and relation \eqref{eq1.37} when applied with $ (y_{\chi}^{\epsilon},v_{\chi}^{\varepsilon},\xi_{\chi}^{\varepsilon}) $ give : \\
	
	\begin{center}
	$ \forall \chi \in L^{p}(\Omega) ,~~$ $\mid\mid \chi \mid\mid_{L^{p}(\Omega)} \leq 1,~~~$  $ \rho \langle\ q_{\varepsilon}, \chi \rangle \leq C_{\chi, \varepsilon} \leq C. $
	\end{center}
	\end{Proof}
	Then we may pass to the limit in the penalized optimality system and obtain the following result. 
	\begin{Theorem} 
	
	Assume $( \mathcal{H}_{1}) $ and $( \mathcal{H}_{2}), $ if  $~(y_{\alpha},v_{\alpha},\xi_{\alpha})$ is a solution of $( \mathcal{P}^{\alpha})$, then Lagrange multipliers $(q_{\alpha},r_{\alpha})\in L^{p^{'}}(\Omega)\times \R^{+} $ exist, such that 
	
	\begin{equation} \label{eq1.38}	
	 \forall y\in \tilde{K}, ~~ [A+g^{'}(y_{\alpha})](y-y_{\alpha})\in L^{p}(\Omega) \quad  \langle\ p_{\alpha}+q_{\alpha},[A+g^{'}(y_{\alpha})](y-y_{\alpha}) \rangle  + r_{\alpha}\langle\ \frac{\alpha}{(y_{\alpha}+\alpha)^{2}}, y-y_{\alpha} \rangle \geq 0,  
	\end{equation}
	
	\begin{equation} \label{eq1.39}	
	 \forall v \in U_{ad}, ~~  v-v_{\alpha} \in L^{p}(\Omega) \quad    \langle\ \nu(v_{\alpha}-v_{d})-q_{\alpha}, v-v_{\alpha} \rangle \geq 0,
		\end{equation}
	\begin{equation} \label{eq1.40}	
	 \forall \xi \in V_{ad},  ~~ \xi-\xi_{\alpha} \in L^{p}(\Omega)\quad  \langle\ \frac{r_{\alpha}\alpha}{(\xi_{\alpha}+\alpha)^{2}}-q_{\alpha}, \xi-\xi_{\alpha} \rangle \geq 0, 
		\end{equation}
	\begin{equation} \label{eq1.41}	
	 r_{\alpha} \biggl( \langle\ 1, \frac{y_{\alpha}}{y_{\alpha}+\alpha}+\frac{\xi_{\alpha}}{\xi_{\alpha}+\alpha} \rangle -Area(\Omega) \biggr) =0,
		\end{equation}
	
where $ p_{\alpha} $ is given by \eqref{eq1.24}.	
	\end{Theorem} 
	\subsection{ Sufficient condition for $ (\mathcal{H}_{2} )$ with p=2. }
	
	In this subsection we give an assumption dealing with $ (y_{\alpha}, v_{\alpha}, \xi_{\alpha}) $ where $ \varepsilon $ does not appear. We choose $ p=2 $ because it is the most useful case. We always assume that $ g^{'} $ is locally lipschitz continuous (for example $ g $ is $ \mathcal{C}^{2}$), and we set the following 
	\begin{center}
	$  \exists \rho >0, ~\exists v_{0} \in \text{Int}_{\infty}(U_{ad}),~ \forall \mathcal{X}\in L^{2}(\Omega)$ such that $  \mid\mid \mathcal{X} \mid\mid_{L^{2}(\Omega)} \leq 1,$ \\
	  \vspace{0.2cm}
	$\exists(y_{\mathcal{X}},\xi_{\mathcal{X}}) \in \tilde{K} \times V_{ad}$ (uniformly bounded by a constant M independent of $\mathcal{X}  $ ),$ ~~~~ $ \hspace{1mm} $(\mathcal{H}_{3})  $ \V \\
	  \vspace{0.2cm}
	such that $ Ay_{\mathcal{X}}+g^{'}(y_{\alpha})y_{\mathcal{X}} = f+\omega_{\alpha}+v_{0}+\xi_{\mathcal{X}}-\rho\mathcal{X} $ in $ \Omega $.  
	\end{center}

	\begin{Proposition}
	If $ g^{'} $ is locally lipschitz continuous then ($ \mathcal{H}_{3}) \Longrightarrow (\mathcal{H}_{2}$).
	\end{Proposition}

	\begin{Proof}
	
	We have seen that $ \mid\mid y_{\varepsilon}-y_{\alpha} \mid\mid_{\infty} \to 0, $ 
	 $\mid\mid g^{'}(y_{\varepsilon})-g^{'}(y_{\varepsilon})\mid\mid_{\infty} \to 0 $ and $ \mid\mid \omega_{\varepsilon}-\omega_{\alpha} \mid\mid_{\infty} \to 0.$ \\
	Let be $\mathcal{X} \in L^{2}(\Omega)$ such taht $ \mid\mid \mathcal{X} \mid\mid_{2} \leq 1 $ and $(y_{\mathcal{X}},v_{0},\xi_{\mathcal{X}}) \in \tilde{K} \times \text{Int}_{\infty}(U_{ad}) \times V_{ad} $ given by ($ \mathcal{H}_{3}$). As $ v_{0} \in \text{Int}_{\infty}(U_{ad}) $, there exists $ \rho_{0} >0 $ such that $ \mathcal{B}_{\infty}(v_{0}, \rho) \subset U_{ad} $. As $ y_{\mathcal{X}} $ is bounded by $ M $, then for $ \varepsilon $ small enough (less than some $ \varepsilon_{0} >0 $), we get 
	\begin{center}
	  $ \mid\mid \omega_{\alpha}-\omega_{\varepsilon}+(g^{'}(y_{\varepsilon})-g^{'}(y_{\alpha}))y_{\mathcal{X}} \mid\mid_{\infty}$ $ \leq \mid\mid   \omega_{\alpha}-\omega_{\varepsilon}  \mid\mid_{\infty} +\mid\mid g^{'}(y_{\varepsilon})-g^{'}(y_{\alpha}) \mid\mid_{\infty} \mid\mid  y_{\mathcal{X}} \mid\mid_{\infty} \leq \rho_{0},   $
	  \end{center}  
	
therefore $ v_{\mathcal{X}}^{\varepsilon} = v_{0}+(g^{'}(y_{\varepsilon})-g^{'}(y_{\alpha}))y_{\mathcal{X}}+\omega_{\alpha}-\omega_{\varepsilon} $ belongs to $ U_{ad} $ and 
\begin{center}
	 $ \mid\mid    v_{\mathcal{X}}^{\varepsilon} \mid\mid_{2}~ \leq ~	 \mid\mid    v_{0} \mid\mid_{2} +\mid\mid   \omega_{\alpha}-\omega_{\varepsilon} \mid\mid_{2} +  \mid\mid    (g^{'}(y_{\varepsilon})-g^{'}(y_{\alpha}))y_{\mathcal{X}} \mid\mid_{2} ~ \leq  ~C,$ 
	\end{center}	
	$ v_{\mathcal{X}}^{\varepsilon} $ is $ L^{2}\text{-bounded} $ independently of $ \mathcal{X} $ and $ \varepsilon $. Now, we set $ y_{\mathcal{X}}^{\varepsilon}=y_{\mathcal{X}} \in \tilde{K} $ and $ \xi_{\mathcal{X}}^{\varepsilon}=\xi_{\mathcal{X}} \in V_{ad} $ to obtain
	 
\begin{equation} \label{eq1.42}
\begin{split}
Ay_{\mathcal{X}}^{\varepsilon}+g^{'}(y_{\varepsilon})y_{\mathcal{X}}^{\varepsilon} & = Ay_{\mathcal{X}} +g^{'}(y_{\alpha)})y_{\mathcal{X}}+(g^{'}(y_{\varepsilon})-g^{'}(y_{\alpha}))y_{\mathcal{X}}  \\
 & = f+\omega_{\alpha}+v_{0}+\xi_{\mathcal{X}}-\rho \mathcal{X}+(g^{'}(y_{\varepsilon})-g^{'}(y_{\alpha}))y_{\mathcal{X}} \\
  & = f+\omega_{\varepsilon}+v_{0}+ (g^{'}(y_{\varepsilon})-g^{'}(y_{\alpha}))y_{\mathcal{X}}+\omega_{\alpha}-\omega_{\varepsilon} +\xi_{\mathcal{X}} -\rho \mathcal{X} \\
  & = f+\omega_{\varepsilon}+v_{\mathcal{X}}^{\varepsilon}+\xi_{\mathcal{X}}^{\varepsilon}-\rho \mathcal{X}. 
\end{split}
\end{equation}

	We can see that that ($ \mathcal{H}_{2})$ is satisfied.
	\end{Proof}

	An immediate consequence is the following Theorem: we get the existence of Lagrange multipliers:

	\begin{Theorem} 
	Let $ (y_{\alpha}, v_{\alpha}, \xi_{\alpha}) $ be a solution of ($ \mathcal{P^{\alpha}}$) and assume ($ \mathcal{H}_{1}$) and ($ \mathcal{H}_{3}$); then Lagrange multipliers \\$( q_{\alpha}, r_{\alpha}) $ $ \in L^{2}(\Omega) \times \R^{+} $ exist, such that 
	 
	\begin{equation} \label{eq1.43}
	 \forall y\in \tilde{K},  \quad     \langle\ p_{\alpha}+q_{\alpha},[A+g^{'}(y_{\alpha})](y-y_{\alpha}) \rangle  + r_{\alpha}\langle\ \frac{\alpha}{(y_{\alpha}+\alpha)^{2}}, y-y_{\alpha} \rangle \geq 0, 
	\end{equation}
	
	\begin{equation} \label{eq1.44}
	 \forall v \in U_{ad},    \quad   \langle\ \nu(v_{\alpha}-v_{d})-q_{\alpha}, v-v_{\alpha} \rangle \geq 0,  
	\end{equation}
	
	\begin{equation} \label{eq1.45}
	 \forall \xi \in V_{ad},   \quad   \langle\ \frac{r_{\alpha}\alpha}{(\xi_{\alpha}+\alpha)^{2}}-q_{\alpha}, \xi-\xi_{\alpha} \rangle \geq 0,
	\end{equation}
	
	\begin{equation} \label{eq1.46}
	 r_{\alpha} \biggl( \langle\ 1, \frac{y_{\alpha}}{y_{\alpha}+\alpha}+\frac{\xi_{\alpha}}{\xi_{\alpha}+\alpha} \rangle -Area(\Omega) \biggr) =0,
	\end{equation}
	
where $ p_{\alpha} $ is given by \eqref{eq1.24}.	

	\end{Theorem}

	\begin{Proof}
	We take $ v_{0} = -(f+\omega_{\alpha}) $ to ensure $( \mathcal{H}_{3})$. Let $ \mathcal{X} \in L^{2}(\Omega) $ such that  $ \mid\mid \mathcal{X} \mid\mid_{L^{2}(\Omega)} \leq 1. $ \\ 
	We set  $\xi_{\mathcal{X}}=\mathcal{X}^{+}+\mathcal{X}^{-}= \vert \mathcal{X} \vert \geq 0,$ where $ \mathcal{X}^{+} $ = $ max(0, \mathcal{X}) ~$ and  $~ \mathcal{X}^{-} = $ max$(0, -\mathcal{X})$. As $ \mid\mid \mathcal{X} \mid\mid_{L^{2}(\Omega)} \leq 1,$ it is clear that $\xi_{\mathcal{X}} \in V_{ad}  $.
	
	\newpage
	
	 Let $y_{\mathcal{X}}  $ be the solution of  
	\begin{center}
	$  [A+g^{'}(y_{\alpha})]y_{\mathcal{X}}=$  $  \xi_{\mathcal{X}}-\mathcal{X}$ = $  2\mathcal{X}^{-} \geq 0 \quad (a.e.),~ y\in H_{0}^{1}(\Omega),$ 
	\end{center}
	thanks to the properties of  $ [A+g^{'}(y_{\alpha})] $ and the maximum principale, then $ y_{\mathcal{X}}  \geq 0$ a.e. in $ \Omega $. Therefore $ y_{\mathcal{X}} \in \tilde{K}$ and $ (\mathcal{H}_{3}) $ is satisfied (with $  \rho = 1$ ). The optimality system follows and we have proved that the multiplier $ q_{\alpha} $ is a $ L^{2}(\Omega) $-function. 
	\end{Proof}
	\begin{Corollary}
	If $ g $ is linear and $ -f \in  \text{U}_{ad},$ the conclusions of Theorem  are valid.
	\end{Corollary}
	\begin{proof}
	If $ g $ is linear, we use the same proof as the one of Theorem $  6.4$ to bound $ q_{\varepsilon} $ in $ L^{2}(\Omega) $. It is sufficient that $ -f \in\text{U}_{ad}. $
	\end{proof}
	\begin{Remark}
We may choose for example $U_{ad} =[a, b] $ with $ a+3+\alpha \leq b-\alpha, ~ \alpha >0, ~ -b+\alpha \leq -a-3-\alpha$ and $ g(x)=- \frac{1}{1+x^{2}} $. In this cas  $ 0 \leq \omega_{\alpha} \leq 3  $ so that $ -(f+ \omega_{\alpha})\in  [a+\alpha, b-\alpha] \subset \text{Int}_{L^{\infty}}(U_{ad}). $ 
	\end{Remark}
 \vspace{1cm}

Next we describe numerical experiments that were carried out by means of AMPL languages   \cite{AMPL}, with the IPOPT solver  \cite{IPOPT} ("Interior Point OPTimizer"), KNITRO solver  \cite{KNITRO} ("Nonlinear Interior point Trust Region Optimization" ) and SNOPT solver  \cite{SNOPT}("Sparse Nonlinear OPTimizer"). 
	  \section{NUMERICAL RESULTS}
In this section, we report on some experiments considering a 2D-example. For two different smoothing functions, we present some numerical results using the IPOPT nonlinear programming algorithm on AMPL  \cite{AMPL}  optimization plateform. Our aim is just to verify the qualitative numerical efficiency of our approach. The discretization process was based on finite difference schemes with a $ N \times N $ grid and the size of the grid is given by $ h=\frac{1}{N} $ on each side of the domain.  \\
We take $ \Omega = ]0,1[ \times ]0,1[ \subset \R^{2}, ~A:= \bigtriangleup $ the Laplacian operator $ ( \bigtriangleup y=\frac{\partial^{2}y}{\partial x_{1}^{2}}+\frac{\partial^{2}y}{\partial x_{2}^{2}}).$ We fixe the tolerance to  $ \text{tol}=10^{-3}$ and the smoothing parameter to $ \alpha=10^{-3}.$ 
In our experiments, we use the two following functions  
	\begin{center}
	$  \theta_{\alpha}^{1}(x)=\frac{x}{x+\alpha},$\\
	\vspace{0.5cm}
	$  \theta_{\alpha}^{\text{log}}(x)= \frac{\text{log}(1+x)}{\text{log}(1+x+\alpha)}$.
\end{center}
	\subsection{Description of the example:}
	 We set $U_{ad}= L^{2}(\Omega),~ \nu = 0.1, ~z_{d}=1,~v_{d}=0, ~g(y)=y^{3}.$
	 \vspace{0.2cm}
	\begin{center}
	$  
	f(x_{1},x_{2}) = \left\{
    \begin{array}{ll}
        200[2x_{1}(x_{1}-0.5)^{2}-x_{2}(1-x_{2})(6x_{1}-2)] & \mbox{if } x_{1} \leq 0.5, \\
        200(0.5-x_{1}) & \mbox{else,}
    \end{array}
\right.
$
	\end{center}
 \vspace{0.5cm}
	\begin{center}
	$  
	\psi(x_{1},x_{2}) = \left\{
    \begin{array}{ll}
        200[x_{1}x_{2}(x_{1}-0.5)^{2}(1-x_{2})] & \mbox{if } x_{1} \leq 0.5, \\
        200[(x_{1}-1)x_{2}(x_{1}-0.5)^{2}(1-x_{2})] & \mbox{else.}
    \end{array}
\right.
$
	\end{center}

	\begin{figure}[H]
\begin{minipage}[H]{.46\linewidth}
     \begin{center}
             \includegraphics[width=6cm]{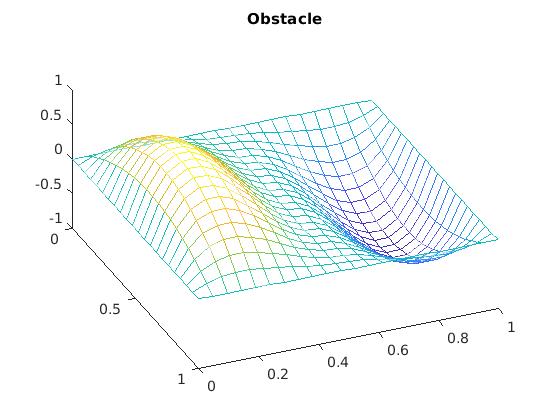}
         \end{center}
   \end{minipage} \hfill
   \begin{minipage}[H]{.46\linewidth}
    \begin{center}
            \includegraphics[width=6cm]{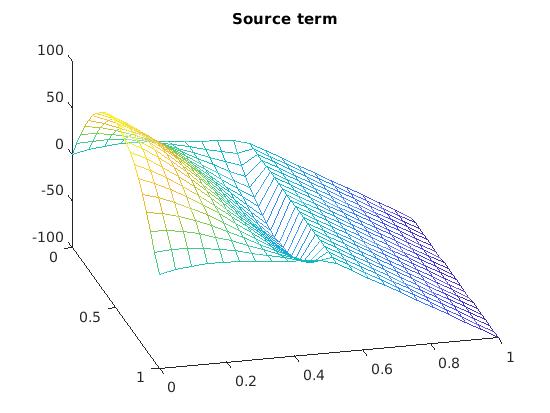}
        \end{center}
 \end{minipage}
       \caption{Data of the considered example}
	\end{figure}

\begin{figure}[h]
\begin{minipage}[c]{.46\linewidth}
     \begin{center}
             \includegraphics[width=6cm]{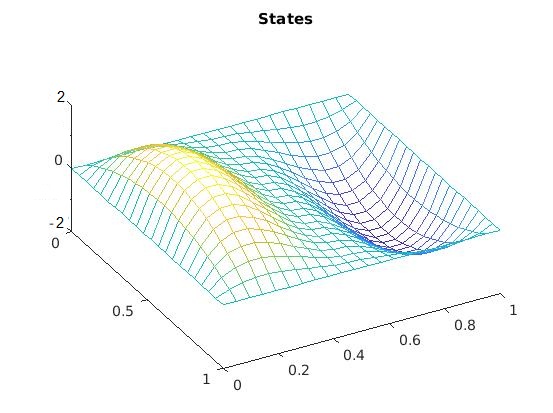}
         \end{center}
   \end{minipage} \hfill
   \begin{minipage}[c]{.46\linewidth}
    \begin{center}
            \includegraphics[width=6cm]{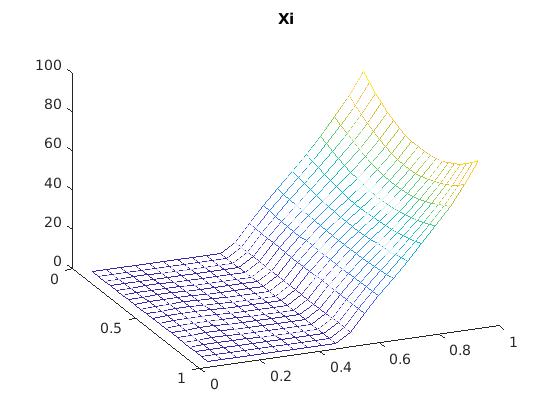}
        \end{center}
 \end{minipage}
       
\begin{minipage}[c]{.46\linewidth}
     \begin{center}
             \includegraphics[width=6cm]{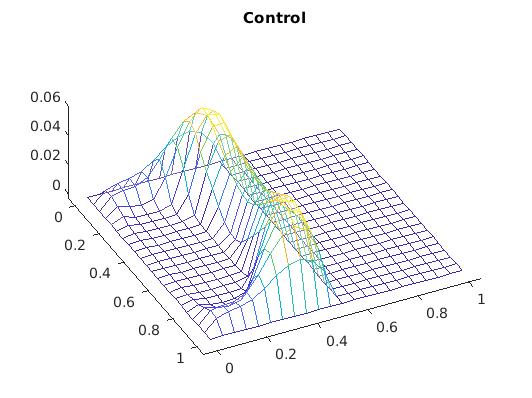}
         \end{center}
   \end{minipage} \hfill
   \begin{minipage}[c]{.46\linewidth}
    \begin{center}
            \includegraphics[width=6cm]{Obstacle.jpg}
        \end{center}
 \end{minipage}
       \caption{Optimal solution using the $ \theta_{\alpha}^{1}$, N=20 and $ \alpha=10^{-3} $ }
	\end{figure}

\subsection{Details of the numerical tests}
\subsubsection{Numerical simulation results using IPOPT solver}
In our experiments we made a logarithmic scaling for these two functions to bound their gradients. Each constraint 
\begin{center}
	$ \theta_{\alpha}((y-\psi)_{i,j})+\theta_{\alpha}(\xi_{i,j})\leq 1 $
\end{center}
is in fact replaced by the following inequality 
\begin{center}
	$ \alpha^{2}~\text{ln} \left(  \frac{\alpha}{(y-\psi)_{i,j}+\alpha}+\frac{\alpha}{\xi_{i,j}+\alpha}\right)\geq 0,~~~~~~~~~~~ ~~~~~~~~~~ 0 \leq i,j \leq N+1 $
\end{center}
in the case of the $  \theta_{\alpha}^{1} $ function and 

\begin{center}
	$ \alpha~\text{ln} \left( 2-\left(   \frac{\text{log}(1+(y-\psi)_{i,j})}{\text{log}(1+(y-\psi)_{i,j}+\alpha)} +\frac{\text{log}(1+\xi_{i,j})}{\text{log}(1+\xi_{i,j}+\alpha)} \right) \right)\geq 0, ~~~~~~~ 0 \leq i,j \leq N+1  $
\end{center}
	in the case of the $  \theta_{\alpha}^{\text{log}}$.\\ 	
 This scaling technique was proposed and used in  \cite{Birbil} to avoid numerical issues. The two following tables give in view of  exemple 7.1 and for different values of the  parameter $ \alpha $, the complementarity error, the state equation error  and the solution obtained  when using each of the two smoothing functions.

\begin{table}[h]
	\centering
	\begin{tabular}{|C{2cm}||C{4cm}|C{4cm}|C{4cm}|}
		\hline  $  \alpha$ & $\mid\mid  Ay-g(y)-f-v-\xi \mid\mid_{2} $ &   $\langle y-\psi, \xi \rangle /N^{2}$ &  $ \text{Obj} $  \\
		\hline   $  0.1$ & $9.7453e-14 $ & $ 9.911e-3$ & $ 2.8420e+02  $  \\
		\hline  $  10^{-2}$ & $ 8.2725e-14 $ & $9.998e-5 $ & $  2.8564e+02  $  \\
		\hline   $  10^{-3}$ & $8.73843e-14  $ & $ 9.675e-7 $ & $ 2.8572e+02 $ \\
		\hline  $ 10^{-4}$ & $1.75849e-06  $ & $5.243e-09 $ & $2.8573e+02 $  \\
		\hline 
	\end{tabular}
	\caption{Using the $ \theta^{1}_{\alpha} $ smoothing function -Example 7.1- N=20 }
\end{table}

	\begin{table}[h]
		\centering
		\begin{tabular}{|C{2cm}||C{4cm}|C{4cm}|C{4cm}|}
			\hline  $  \alpha$ & $\mid\mid  Ay-g(y)-f-v-\xi \mid\mid_{2} $ &   $\langle y-\psi, \xi \rangle /N^{2} $ &  $\text{Obj} $  \\
			\hline   $  0.1$ & $9.00735e-14  $ & $4.943e-3 $ & $2.8455e+02  $  \\
			\hline  $  10^{-2}$ & $8.89491e-14  $ & $6.013e-5  $ & $2.8564e+02  $  \\
			\hline   $  10^{-3}$ & $9.22557e-14  $ & $5.892e-7  $ & $2.8572e+02  $ \\
			\hline  $ 10^{-4}$ & $3.47932e-06  $ & $8.124e-8 $ & $2.8573e+02  $  \\
			\hline 
		\end{tabular}
		\caption{Using the 	$  \theta_{\alpha}^{\text{log}}$ smoothing function -Example 7.1- N=20 }
	\end{table}
	\subsubsection{Numerical comparisons using different solvers : IPOPT  \cite{IPOPT} , KNITRO \cite{KNITRO} and SNOPT \cite{SNOPT}  }

		\begin{table}[h]
		\centering
		\begin{tabular}{|C{4cm}||C{4cm}|C{4cm}|C{4cm}|}
			\hline  Solver & SNOPT &  KNITRO &  IPOPT  \\
			\hline   $ \mid\mid  Ay-g(y)-f-v-\xi \mid\mid_{2} $ & $1.06926e-12  $ & $6.00605e-14 $ & $6.71515e-14
			 $  \\
			\hline  $   \langle y-\psi, \xi \rangle /N^{2} $ & $8.7111e-5 $ & $8.7277e-5 $ & $8.70844e-5 $  \\
			\hline   $  \text{Obj}$ & $1.502476e+02 $ & $1.502476e+02   $ & $1.502476e+02  $ \\
			\hline  $\text{Iter}$ & $26236 $ & $200 $ & $ 198 $  \\
			\hline 
		\end{tabular}
		\caption{Using the 	$  \theta_{\alpha}^{1}$ smoothing function -Example 7.1- N=15 and $\alpha =10^{-2} $ }
	\end{table}

	\begin{figure}[H]
		
			\begin{center}
				\includegraphics[width=12cm,height=7cm]{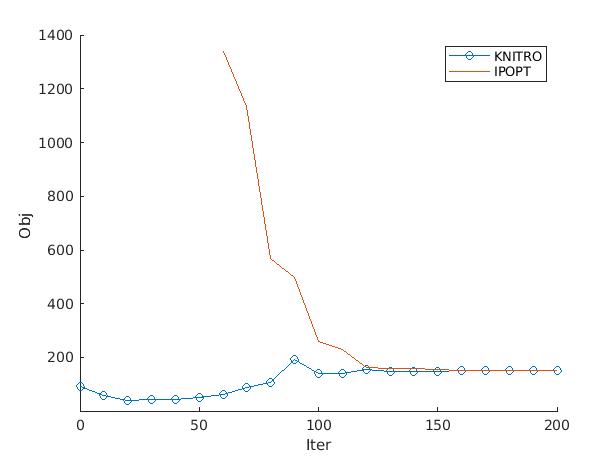}
			\end{center}
		\caption{Numerical comparison of the different numerical solver IPOPT and KNITRO }
	\end{figure}
	We remark that : \\
	The 3 algorithms obtain the same solution and almost the same objective value. This suggests that our approach can be implemented using any standard NLP solver.
	 \newpage
	
\section{Conclusions}	
In this work, we introduced a new regularization schema for optimal control of semilinear elliptic vartional inequalities with complementarity constraints. We proved that Lagrange multipliers exist. The existence of Lagrange multipliers is an important tool to describe and study algorithms to compute the solutions(s) of $(\mathcal{P}^{\alpha})	$ (that are "good approximations" of the original problem $(\mathcal{P})$).
In our numerical experiments, we used several standard NLP solvers and obtain promising results. The next step will be to develop an approach based on our optimality conditions.	


\end{document}